\documentclass[10pt,aps,nofootinbib,twocolumn,superscriptaddress,groupedaddress]{revtex4-1}  
\usepackage{graphicx}  
\usepackage{dcolumn}   
\usepackage{bm}        
\usepackage[percent]{overpic}  
\usepackage{dsfont}
\usepackage{epsfig,amsmath,graphics,verbatim,psfrag, amsfonts, amsthm}
\usepackage{amssymb}

\usepackage{paralist}
\usepackage{xcolor}
\usepackage[font={small,it}]{caption}
\usepackage{subcaption}
\usepackage{booktabs}
\usepackage{tabularx}

\newtheorem{thm}{Theorem}[section]
\newtheorem{cor}[thm]{Corollary}

\newtheorem{lem}[thm]{Lemma}
\newtheorem{ex}[thm]{Example}

\newtheorem{remark}[thm]{Remark}

\numberwithin{equation}{section}

\newcommand{\fa}          {\quad \text{for all } \,}

\newcommand{\nwc}{\newcommand}
\nwc{\blue}[1]{\textcolor{blue}{#1}}

\DeclareMathOperator{\interior}{int}
\DeclareMathOperator{\supp}{supp}
\DeclareMathOperator{\diag}{diag}

\newcommand{\real}{\mbox{${\mathbb R}$}}
\def\R{\mathbb{R}}

\newcommand{\rmd}{\mathrm{d}}

\renewcommand{\epsilon}{\varepsilon}

\newcommand{\be}{\begin{equation}}
\newcommand{\ee}{\end{equation}}
\newcommand{\bea}{\begin{eqnarray}}
\newcommand{\eea}{\end{eqnarray}}
\newcommand{\beann}{\begin{eqnarray*}}
\newcommand{\eeann}{\end{eqnarray*}}
\newcommand{\benn}{\begin{equation*}}
\newcommand{\eenn}{\end{equation*}}

\usepackage{acronym}

\newacro{FRL}[FTRL]{``Follow the Regularized Leader''}
\newacro{MW}{multiplicative weights}
\newacro{MWU}{multiplicative weights update}

\begin{document}

\title{A stochastic variant of replicator dynamics in zero-sum games and its invariant measures}

\author{Maximilian Engel}
\affiliation{Department of Mathematics, Freie Universtit{\"a}t Berlin, 
Arnimallee~6, 14195 Berlin, Germany}

\author{Georgios Piliouras}
\affiliation{DeepMind, 14-18 Handyside Street, London N1C 4DN, UK}


\begin{abstract}
We study the behavior of a stochastic variant of replicator dynamics in two-agent zero-sum games. 
Inspired by the well studied deterministic replicator equation, this model gives arguably a more realistic description of real world settings and, as we 
demonstrate here, exhibits totally distinct behavior when compared to its deterministic counterpart which is known to be recurrent.
In more detail, we characterize the statistics of such systems by their invariant measures which can be shown to be entirely supported on the boundary of the space of mixed strategies. Depending on the noise strength we can furthermore characterize these invariant measures by finding accumulation of mass at specific parts of the boundary. In particular, regardless of the magnitude of noise, we show that any invariant probability measure is a convex combination of Dirac measures on pure strategy profiles, which correspond to vertices/corners of the agents' simplices.  Thus, in the presence of stochastic perturbations, even in the most classic zero-sum settings, such as Matching Pennies, we observe a stark disagreement between the axiomatic prediction of Nash equilibrium and the evolutionary emergent behavior derived by an assumption of stochastically adaptive, learning agents. 
\end{abstract}

\maketitle

\section{Introduction}




Zero-sum games are arguably amongst the most well studied settings in game theory and economics. Dating back to seminal work by von Neumann~\cite{Neumann1928} on his famous Minimax Theorem, the study of zero-sum games effectively marked the birth of game theory itself. More importantly, it is the archetypal success story of the field. This is neatly captured by a well known quote of Nobel-prize winning economist Robert Aumann~\cite{eatwell1987new}:
``Strictly competitive games constitute one of the few areas in game theory, and indeed in the social sciences, where a fairly sharp, unique prediction is made".   
Indeed, almost all zero-sum games have a unique equilibrium\footnote{Formally, within the set of all zero-sum games, the complement of the set of zero-sum games with unique equilibrium is closed and has Lebesgue measure zero~\cite{van1991stability}.} and, 
even when multiple equilibria exist, the expected utility of each agent (their value) is uniquely defined. 
These admirable properties are nicely complemented by the computational tractability of solving for Nash equilibria in zero-sum games due to their fundamental connection to linear programming  
\cite{dantzig1951proof}. 
  However, from a computational perspective a more pressing question is whether natural decentralized learning dynamics converge to equilibria 
 robustly, i.e., even in the presence of stochastic perturbations that reflect misspecifications to the underlying game. This is particularly true for real world applications where the assumption of perfect competition is an idealized abstraction and the actual day-to-day payoffs are subject to stochastic perturbation due to background/environmental uncertainty (e.g. duopoly competition where each firm tries to maximize the number of its customers but where the total number of customers could fluctuate). Our goal in this paper is to make inroads in our understanding of a specific well motivated class of stochastic  learning dynamics.
 
Naturally, the question of whether the behavior of learning dynamics corresponds well to Nash equilibria in zero-sum games is similarly a classic object of study dating back to seminal work by Brown and Robinson in the 1950s~\cite{Brown1951,Robinson1951}.
 Historically, economists, game theorists and optimization theorists alike, when studying the behavior of learning dynamics in games, have focused on the time-average behavior of the dynamics, proving that typically both the time-average strategies as well as the time-average utilities of both agents converge to their respective equilibrium values~\cite{freund1999adaptive,Nisan:2007:AGT:1296179}.  
 A critical tool for showing such positive time-average convergence results is to establish that such dynamics are (approximately) regret minimizing as any such pair of algorithms is guaranteed to converge to (approximate) equilibria in a time-average sense. This widely applicable technique has led to a proliferation of such positive convergence results for large classes of dynamics~\cite{Cesa06,Shalev2012}. In contrast, the day-to-day behavior of learning dynamics in zero-sum games, even of those dynamics that minimize regret, 
 had not received similar levels of attention up until recently when such questions where invigorated due to their connection to Machine Learning applications such as Generative Adversarial Networks (GANs)~\cite{goodfellow2014generative}. 

A natural starting point for kick-starting the investigation into the day-to-day behavior of regret-minimizing 
learning dynamics in zero-sum games, has been the model of replicator dynamics, arguably the most well studied dynamics in evolutionary game theory~\cite{Weibull}. The replicator structure is  the continuous-time-analogue of Multiplicative Weights Update (MWU), a ubiquitous meta-algorithm 
 with strong regret guarantees~\cite{Arora05themultiplicative}.  Despite these strong regret properties of replicator dynamics~\cite{sorin2009exponential},
 its day-to-day behavior in zero-sum games does not converge to Nash equilibria but instead ``cycles"  around at a constant Kullback-Leibler divergence from the Nash equilibrium as long as the equilibrium lies in the interior of the simplex~\cite{piliouras2014optimization}. If the  game does not admit interior equilibria then the dynamics converge to the subspace of strategies defined by the equilibrium of maximum support.\footnote{Since in zero-sum games the set of Nash equilibria is a convex polytope, the notion of equilibrium of maximum support is well defined. For example, take the equilibrium that corresponds to the barycenter of the Nash polytope.} In that space the behavior of the dynamics are once again recurrent. 
 In fact, this result generalizes for all continuous-time variants of the well known regret-minimizing family of Follow-the-Regularized Leader (FTRL) dynamics~\citep{GeorgiosSODA18}. 
  Behind this regularity in behavior lies the fact that these dynamics are--after an appropriate reparametrization--Hamiltonian \citep{Hofbauer96,bailey2019multi}.
 As in the case of the movement of, e.g., a perfect, idealized pendulum, there is a lot of hidden structure (e.g., constants of motion) in the resulting deterministic dynamics that allows for a pretty thorough understanding of these simplified models.

 Unfortunately, this level of regularity largely disappears when we move away from deterministic models towards more \textit{realistic, stochastic} ones.  Our goal will exactly be to perform a careful analysis of the properties of such stochastic models. Clearly, from the perspective of applications, be they in economics, game theory or machine learning the move towards stochastic models seems necessary. For example, in economics and game theory it is rather natural to consider the possibility of stochastic system shocks that perturb the agents' beliefs and behavior~\cite{foster,mertikopoulos2010emergence}; in fact, using stochastic methods in machine learning, e.g., based on batch learning, is the norm as it allows for much more efficient implementation of standard optimization techniques~\cite{jin2017escape,vlatakis2019efficiently}. 
  On top of the practical necessity of studying such stochastic models, 
what is even more remarkable about the case of replicator dynamics in zero-sum games is that its predicted behavior is particularly brittle under stochastic perturbations. Since the recurrent behavior is based on the existence of a constant of motion for the deterministic dynamics, the stochastic perturbations will almost certainly introduce  
some flux that will radically alter the long term behavior of the dynamics.
For similar reasons standard techniques based on stochastic approximation theory~\cite{benaim2006stochastic} are not very informative, since in the deterministic system all states are internally chain transitive and thus different ideas as needed. 
Thus some critical questions emerge: 

\medskip
\textit{Can we analyze natural stochastic variants of replicator dynamics in general zero-sum games? If so, what is the emergent long term behavior? Do the dynamics stabilize, cycle, or do they produce a different type of behavior altogether?}
\medskip





{\bf Our results.} 
We investigate a model of replicator dynamics  
 perturbed by multiplicative Brownian noise with a diffusion matrix  
  satisfying two typical main features; firstly, an orthogonality property to keep the simplex invariant under the stochastic dynamics and, secondly, a mild form of ellipticity. 
 The first property requires that any column of the diffusion matrix at a position vector $x$ is orthogonal to $x$ implying that the noise preserves the total (probability) mass of the strategies, i.e.~does not destroy the interpretation as mixed strategies. By mild form of ellipticity we mean uniform lower bounds on pairwise combinations of row vectors in the diffusion matrix in terms of their Euclidean norms added up.
  We study this model in zero-sum games and show that all invariant measures are combinations of Dirac measures at the corners, i.e.~agents asymptotically play mixed strategies with probability zero; that is, even if the game has a unique interior Nash equilibrium, as e.g.~in the standard example of Matching Pennies, the dynamics behave in antithetical way to the prediction of the unique Nash equilibrium. The agents are almost never mixing but instead spend almost all of their time playing effectively pure, i.e. non-randomizing, strategies. 
Furthermore, the supports of attracting invariant measures, i.e.~the ones where trajectories from the interior converge to, are distributed over all corners when the Nash equilibrium of the deterministic game is fully mixed; otherwise, at least for sufficiently small noise, the attracting invariant measures are supported within the part of the boundary where the Nash equilibrium of maximal support is located. The main technique of proof is using Lyapunov function theory for stochastic processes, as firstly summarized by Khasminskii \cite{Khasminskii80}.
Throughout the paper, we discuss a particular choice of the diffusion matrices 
such that the model obtains the interpretation of symmetric uncertainty around the deterministic utilities. We exemplify the general results with this particular choice of uncertainty at the hand of the classical example of Matching Pennies, i.e.~the two-player zero-sum game with one player winning if two secretly turned pennies match, in terms of heads or tails, and the other one winning if they do not match; we discuss this example in the case of fully and also in cases with a partially (i.e. non-fully) mixed Nash equilibrium.

\section{Related Work}

{\bf Replicator dynamics and stochastic variants in game theory.}
As already mentioned, replicator dynamics is a basic model of evolutionary game theory and learning in games. 
There have been several variants of replicator dynamics in terms of stochastic differential equations (SDEs) studied in the last three decades, and, without being able to discuss an exhaustive list here due to space constraints, we will briefly recall the works most strongly related to our approach.
The first stochastic model was introduced in \cite{foster} where the stochastic component was introduced directly in the replicator differential equation. The SDE reads
\[
dx_i = x_i [(Ax)_i dt - x^{\top} Ax dt + \sigma \Gamma(x) dW(t)_i],
\]
\noindent
where  $W(t)$ is standard Brownian motion, and $\Gamma(x)$ is continuous in $x$ and has the property that
$x^{\top}\Gamma(x) = \vec{0}$.  
Our model will be precisely the bimatrix, two-agent version of this model.
Given this stochastic model, Foster and Young study the behavior  of the system as the intensity of the noise $\sigma$ goes to zero and show that in the case of single agent evolutionary games the system selects among the different evolutionary stable states (ESS), when they exist. Our analysis will give a different perspective, characterizing the stochastic dynamics and the related invariant measures for fixed noise strengths, also dealing with the intricacies of the bimatrix structure.

A different stochastic model was introduced by \cite{fudenberg1992evolutionary}, which has then been further studied in \cite{BenaimHofbauerSandholm, BenaimSchreiberAtchade, HofbauerImhof, Imh05}. 
Once again this model corresponds to a one-agent evolutionary game inspired by biological single population dynamics. This model is related to that of Foster and Young, but exhibits a boundary behavior 
that appears to be more realistic from a biological perspective. 
The SDE is derived by introducing stochastic jumps in the ordinary differential equation (ODE) that describes the evolution of the total population of different subspecies before it is normalized into a probability distribution to derive the replicator ODE. The resulting SDE is of the form: 
\begin{align}\label{eq:Imhof_model}
\rmd \mathbf{X(t)} = &\left( \diag (X_1(t), \dots, X_n(t))  - \mathbf{X(t)X(t)^{\top}}\right)\nonumber\\ 
&\bigg(\mathbf{[A - diag(\sigma_1^2, \dots, \sigma_n^2)]X(t)} \, \rmd t \nonumber\\
&+ \diag(\sigma_1, \dots, \sigma_n) \,  \rmd W(t) \bigg).
\end{align} 


Modulo the difference between one and two agent games, this is similar to the specific example of our model if we replace matrix $A$ with  $A - diag(\sigma_1^2, \dots, \sigma_n^2)$.
In  \cite{Imh05}, it is shown that under suitable conditions, if an ESS exists, then X(t) is recurrent and the stationary
distribution concentrates mass in a small neighborhood of the ESS.  The survival of dominated strategies is considered along with sufficient conditions for asymptotic stochastic stability of equilibria.

The same model is also studied in \cite{HofbauerImhof} where the previous work is extended by establishing an averaging principle that relates time averages
of the process and Nash equilibria of a suitably modified game. 
In addition to also studying necessary and sufficient
conditions for the stochastic stability of pure equilibria, the authors provide
a sufficient condition for transience in terms of mixed equilibria and
definiteness of the payoff matrix. 
Such transient behavior is, in fact, the typical behavior of our model and will be characterized by invariant measures on the boundary in the present paper.

 Other results on this particular model consider robustness of permanence and impermanence under these stochastic perturbations \cite{BenaimHofbauerSandholm}, i.e.~the concentration of stationary measures around hyperbolic attractors (as opposed to the perturbation of elliptic level sets, as in our case).  Further generalizations of stochastic persistence, also applying to this model, can be found in \cite{BenaimSchreiberAtchade}. We also point the reader to related treatments of more general, not necessarily compact state space models that cover stochastic persistence in replicator dynamics \cite{HeningNgyuen2018, HeningNguyenChesson, HeningNguyenSchreiber}.

Another stochastic variant of replicator  has been studied in \cite{mertikopoulos2010emergence}. 
The authors consider multi-player games and introduce a stochastic model 
 inspired by online optimization. They consider mostly congestion games, where the deterministic replicator dynamics are known to converge, and show that the convergence to equilibria survives not just for small enough perturbations but for any  intensity level of stochastic perturbations. 
They do not analyze zero-sum games, where such structure-preserving behavior cannot be observed, as already explained above and discussed in  detail throughout the paper.

{\bf Learning in zero-sum games.}
This area has been the subject of intense study since the 1950s starting with the study of fictitious play dynamics~\cite{Brown1951,Robinson1951}. These are discrete-time dynamics where each agent tracks the frequencies of their opponent's behavior and assumes that in the next time instance they will choose a probability distribution according to this frequency distribution. The analysis of fictitious play showed that the empirical time-average frequency of each player converges to their set of max-min optimal/Nash strategies. Following those seminal works, the standard approach in analyzing numerous different learning dynamics focuses on the convergence of time-averages to equilibrium~\cite{Cesa06,Nisan:2007:AGT:1296179}.
Understanding the time-average behavior, however, does not suffice to provide a full picture.
 Time-average convergence to Nash may hold regardless of whether the system is convergent, recurrent~\cite{piliouras2014optimization,GeorgiosSODA18,bailey2019multi,papadimitriou2019game}, divergent~\cite{BaileyEC18,cheung2018multiplicative,bailey2020finite} or even formally  chaotic~\cite{cheung2019vortices,cheung2020chaos}.
 While a precise exploration of the  behaviors shown by different dynamics is well beyond our scope, it is worth noting the nature of replicator dynamics, since we will be focusing on a stochastic variant of them.
Replicator dynamics is one of the most studied dynamics in evolutionary game theory~\cite{Weibull,Sandholm10} and its flow is a smooth approximation of the well known  multiplicative weights updates algorithm~\cite{Kleinberg09multiplicativeupdates,Arora05themultiplicative}. While it converges to equilibrium in a time-average sense, its trajectories are recurrent in zero-sum games~\cite{piliouras2014optimization}. 
With the recent advent of ML applications organized around zero-sum games such as Generative Adversarial Networks (GANs), much work has focused on developing convergent techniques,~\cite{daskalakis2018training,mertikopoulos2019optimistic,gidel2019a,mescheder2018training,yazici2018unusual}.  Some of these techniques can be interpreted as regularized/perturbed versions of replicator dynamics including mutation-driven dynamics and bounded rationality models~\cite{bauer2019stabilization,perolat2020poincar,leonardos2021exploration,abe2022mutation}.
However, when moving towards more complex zero-sum games, e.g.~non-convex-concave,  which are better models of GANs than normal form bilinear games, numerous negative results show that convergence to meaningful equilibria is probably too ambitious a goal~\citep{adolphs2018local,daskalakis2018limit,vlatakis2019poincare,hsieh2020limits}.  

\section{Replicator dynamics in zero-sum games} \label{deterministiccase}

In this section, we summarize and recall some notation and main insights from deterministic replicator dynamics for zero-sum games  as the key paradigm model of learning in this paper. In this context, we will also prove a lemma concerning the supports of non-interior Nash equilibria and a new notion of anti-equilibria that will help us analyze the stochastic dynamics. Missing proofs can be found in Appendix~\ref{app:determinstic_system}.

%
\subsection{Model}
In the following, we will denote $[n] = \{1,\dots, n \}$. Given $I\subset [n]$, we define $I^c:=[n]\setminus I$. Given any $\mathcal{D} \subset \real^n$, we denote by $\interior (\mathcal{D})$ its interior, by $\overline{\mathcal{D}}$ its closure and by $\partial \mathcal D$ its boundary.

Consider a two-player game with $n$ (resp.~$m$) pure strategies for the first (resp.~second) agent and payoff matrices $\mathbf{A} = (a_{ij})$ and $\mathbf{B} =(b_{ji})$ where $i \in [n]$ 
 and $j \in [m]$. 
The two players choose probability distributions $\mathbf{x}$ and $\mathbf{y}$ encoding mixed strategies for playing the game. Formally, the mixed strategies $\mathbf{x}$, $\mathbf{y}$ lie in the closure of the respective simplices with interiors
\begin{align*}
\Delta_n &:= \{ \mathbf{x} \in (0,1)^n \,:\, x_1 + \dots + x_n =1  \},  \\
\Delta_m &:= \{ \mathbf{y} \in (0,1)^m \,:\, x_1 + \dots + x_m =1  \}.
 \end{align*}  
 In the following, we will consider the domain $\mathcal{D}:= \Delta_n \times \Delta_m$. 
We denote by $u_i$ (resp. $v_j$) the utility of the first (resp. second) agent for playing strategy $i$ (resp. $j$) when their opponent plays mixed strategy $\mathbf{y}$ (resp.~$\mathbf{x}$). 
The respective formulas are $u_i =  \{\mathbf{Ay}\}_i$,  $v_j = \{\mathbf{Bx}\}_j$.

A strategy profile (tuple of strategies) $\mathbf{(p, q)}$ is called a \emph{Nash equilibrium} if no unilateral profitable deviations exist. By linearity it suffices to check deviations to deterministic strategies, i.e. 
\begin{equation} \label{eq:NashEqu_inequ}
\forall i \in [n]:   \{\mathbf{Aq}\}_i \leq  \mathbf{p^{\top} Aq}, \ \forall j \in [m]:  \{\mathbf{Bp}\}_j \leq \mathbf{q^{\top} Bp}.
\end{equation}
A game is called zero-sum if $\mathbf{A}= - \mathbf{B^{\top}}$ , in which case we have for all $(\mathbf x, \mathbf y) \in \overline{\mathcal D}$
\begin{equation}\label{eq:zero_sum}
\mathbf{x^{\top} A y}  + \mathbf{y^{\top} B x} =0.
\end{equation}

 The existence of such equilibria in all zero-sum games follows from von Neumann's celebrated minmax theorem~\cite{neumann}.
In zero-sum games, Nash strategies are also referred to as optimal or max-min strategies.
We define a strategy profile $\mathbf{(p^*, q^*)}$ as an anti-equilibrium if it is a Nash equilibrium of the game with payoff matrices  $-\mathbf{A}$  and $-\mathbf{B}$. Equivalently, this is a Nash equilibrium where each agent interprets the payoff   $\mathbf{x^{\top} Ay}$  (resp.  $\mathbf{y^{\top} B x}$) as costs to be minimized. In algebraic form, we have    
\begin{align} \label{eq:Antiequ_inequ}
\forall i \in [n]:&  \{\mathbf{Aq^*}\}_i \geq  \mathbf{p^{*\top} Aq^*}, \nonumber \\ 
\forall j \in [m]:&  \{\mathbf{Bp^*}\}_j \geq \mathbf{q^{*\top} Bp^*}.
\end{align}
A Nash equilibrium (more generally a strategy profile) $\mathbf{(p, q)}$ is called interior (or fully mixed) if $p_i, q_j>0$ for all $i,j$. The support of a mixed strategy $\mathbf{p}$ (resp. $\mathbf{q}$) corresponds to the set of strategies that are played with positive probability in that strategy, i.e. $supp(\mathbf{p})= \{ i \in [n] : p_i>0\}.$ 
An interior strategy profile has full support.
If a Nash equilibrium is interior then all inequalities in the definition of equilibrium \eqref{eq:NashEqu_inequ} hold as equalities. Thus, any interior Nash equilibrium is also an interior anti-equilibrium, with equality in \eqref{eq:Antiequ_inequ}.
 
In addition to these well-known preliminaries, we provide the following observation concerning the relation between supports of Nash equilibria and anti-equilibria. 
 \begin{lem} \label{lem:supports}
Given a generic\footnote{The complement of the considered set of  games is closed and has Lebesgue measure zero in the space of all zero-sum games.} zero-sum game, let $(p,q)$ and $(p^*,q^*)$ be a Nash equilibrium and anti-equilibrium of maximum support\footnote{ By an equilibrium of maximum support, we mean an equilibrium whose support includes all strategies in the support of any other equilibrium strategy. The existence of such equilibria in all zero-sum games follows from von Neumann's celebrated minmax theorem that implies that the equilibrium strategies for each agent form a non-empty convex polytope~\cite{neumann}.}, respectively. Then either both of them are interior, 
 or, 
$\big(supp(p) \cup supp(q)\big) \setminus \big(supp(p^*) \cup supp(q^*) \big)\neq \emptyset$ and 
$\big(supp(p^*) \cup supp(q^*)\big) \setminus \big(supp(p) \cup supp(q) \big)\neq \emptyset,$
i.e., the supports of equilibria and anti-equilibria are not subsets of each other.
\end{lem}

\begin{proof}
The anti-equilibria are merely Nash equilibria of the zero-sum game defined by payoff matrices $-\mathbf{A}$  and $-\mathbf{B}$. The fact that both the set of Nash equilibria as well as the set of anti-equilibria are convex polytopes follows from von Neumann's celebrated minmax theorem \cite{neumann}. Since equilibrium strategies correspond to convex polytopes the notion of maximum support is well defined; take, e.g., the barycenter of each polytope. If any such equilibrium is fully mixed, then it is clearly also an anti-equilibrium. Suppose not. It suffices to show that it cannot be the case that 
$\big(supp(p) \cup supp(q)\big) \subset \big(supp(p^*) \cup supp(q^*) \big)$\footnote{We can symmetrically apply the argument for the zero-sum game $\big(-\mathbf{A}, -\mathbf{B}\big)$ to derive  that $\big(supp(p^*) \cup supp(q^*)\big) \subset \big(supp(p) \cup supp(q) \big)$ is not possible.} 
 for any zero-sum game $\big(\mathbf{A}, \mathbf{B}\big)$, unless there exists a subgame of this zero-sum game with a continuum of Nash equilibria, which by \cite{van1991stability}  implies that the game $(A,B)$ is non-generic and lies in a closed zero Lebesgue measure set within the set of all zero-sum games.

Indeed, suppose that $supp(p) \subset supp(p^*)$ and  $ supp(p) \subset  supp(q^*)$. Since both the equilibrium $(p,q)$ and anti-equilibrium $(p^*,q^*)$ are of maximum support then they are quasi-strict in the sense that deviating outside their support results to payoffs to the agents that are strictly smaller (larger resp.) than their equilibrium values~\cite{GeorgiosSODA18}.  Hence,  $(p,q), (p^*,q^*)$ are not the same strategy profile. Let us consider the subgame defined by the set of strategies in the support of $(p^*,q^*)$. In this subgame, the anti-equilibrium  $(p^*,q^*)$
is of full support (interior) and thus it is also an (interior) equilibrium of the subgame. By the convexity of the set of Nash equilibria in zero-sum games any convex combination of $(p,q)$ and $(p^*,q^*)$ is also a Nash equilibrium. Thus, the subgame in question has a continuum of equilibria, completing the proof.
\end{proof}

The implication of the above lemma
is that, at in almost all zero-sum games without fully mixed equilibrium,  we can find a pair of equilibrium and anti-equilibrium strategies profiles with at least one strategy played only at equilibrium and at least one strategy played only at anti-equilibrium.

\subsection{Stability and recurrence results for replicator dynamics}
\label{sec:stability_replicator}

The replicator equation is a key model of evolutionary game theory as well as online optimization used to update the mixed strategies of the agents in the direction of improving utility.
In any game, the probability updates for the mixed strategy $\mathbf{x}$ of an agent with a utility vector $\mathbf{u}$ is as follows: 
\begin{equation}  \label{eq:replicator}
\begin{array}{rl}
 \dot x_i &= x_i \left(u_i -  \sum_j x_j u_j \right).
\end{array}
\end{equation}
In the specific case of two player games the replicator equations are given by the ordinary differential equation (ODE)

\begin{equation}  \label{eq:replicator2}
\begin{array}{rl}
 \dot x_i &= x_i \left(\{\mathbf{Ay}\}_i - \mathbf{x^{\top} Ay}\right)\,,\\
 \dot y_j &= y_j \left(\{\mathbf{Bx}\}_j - \mathbf{y^{\top} Bx}\right)\,,
\end{array}
\end{equation}
where $(\mathbf{x}, \mathbf{y}) \in \mathcal{D}= \Delta_n \times \Delta_m$. A Nash equilibrium is always also an equilibrium of the ODE~\eqref{eq:replicator2}.

\subsubsection{Games with Nash interior equilibria}

Here we discuss some variants of known results for (deterministic) replicator dynamics in zero-sum games (adapting arguments from \cite{piliouras2014optimization}). For the sake of completeness as well as to help ease into the analysis of stochastic replicator dynamics we provide the missing proofs in the Appendix~\ref{app:determinstic_system}.


\begin{thm}
\label{thm:constant_of_motion}
Consider the flow of the replicator dynamics when applied to a zero-sum game 
 that has an interior (i.e. fully mixed) Nash equilibrium $(\mathbf{p},\mathbf{q})$.
 Then given any (interior) starting point $x_0\in \mathcal D$,
  the cross entropy
\begin{equation} \label{eq:entropy}
V\big((\mathbf{p},\mathbf{q});(\mathbf{x}(t),\mathbf{y}(t)) \big) = - \sum_i p_i\ln x_i(t) - \sum_j q_j\ln y_j(t)
\end{equation}
   between the Nash equilibrium  $(\mathbf{p},\mathbf{q})$ and the trajectory $(\mathbf{x}(t),\mathbf{y}(t))$ of the system is a \textit{constant of motion}, \textit{i.e.}
    $\frac{\rmd  V\big((\mathbf{p},\mathbf{q});(\mathbf{x}(t),\mathbf{y}(t)) \big)}{\rmd  t} |_{t=t'}= 0$.
 Otherwise, let $(\mathbf{p},\mathbf{q})$ (resp. $(\mathbf{p^*},\mathbf{q^*})$) be a not fully mixed Nash  equilibrium (resp. anti-equilibrium) of maximal support on the boundary $\partial \mathcal D$; then for each starting point $x_0\in \mathcal D$  
 and for all $t'\geq 0$ we have
 $\frac{\rmd  V\big((\mathbf{p},\mathbf{q});(\mathbf{x}(t),\mathbf{y}(t)) \big)}{\rmd  t} |_{t=t'}< 0$
 and 
  $\frac{\rmd  V\big((\mathbf{p^*},\mathbf{q^*});(\mathbf{x}(t),\mathbf{y}(t)) \big)}{\rmd  t} |_{t=t'}>0.$  
\end{thm}
\begin{proof}
   See Appendix~\ref{app:determinstic_system}.
\end{proof}

The quantity $\sum_i p_i\ln p_i - \sum_i p_i\ln x_i(t)$ is always non-negative and it is equal to zero if and only if the distributions $\mathbf{p}$ and $\mathbf{x}(t)$ are equal to each other.
So we have a notion of (non-symmetric) pseudo-distance  which is known as the Kullback-Leiber (K-L) divergence. 
K-L divergence is jointly convex, i.e., both in $\mathbf{x}$ as well as in $\mathbf{p}$ and denoted by $D_{\mathrm{KL}}(\mathbf{p} \| \mathbf{x}(t))$. 
This quantity differs from $ \sum_i p_i\ln x_i(t)$ by a constant. Hence, we establish the following corollary.

\begin{cor}
\label{cor:KL}
Consider the flow of replicator dynamics when applied to a zero sum game with a fully mixed Nash equilibrium $(\mathbf{p},\mathbf{q})$.
  Given any interior 
 starting point $(\mathbf{x}(0),\mathbf{y}(0)) \in \mathcal{D},$  the non-negative quantity 
 $D_{\mathrm{KL}}(\mathbf{p} \| \mathbf{x}(t)) +  D_{\mathrm{KL}}(\mathbf{q} \| \mathbf{y}(t))$   is time invariant.
  Otherwise, let $(\mathbf{p},\mathbf{q})$  (resp.  $(\mathbf{p^*},\mathbf{q^*})$) be a not fully mixed Nash equilibrium (resp. anti-equilibrium) of maximal support  on $\partial D$; then for each starting point $x_0\in \mathcal D$  
 and for all $t'\geq 0$ we have
  $\frac{d\big(D_{\mathrm{KL}}(\mathbf{p} \| \mathbf{x}(t)) +  D_{\mathrm{KL}}(\mathbf{q} \| \mathbf{y}(t))\big)}{dt} |_{t=t'}< 0.$ and
   $\frac{d\big(D_{\mathrm{KL}}(\mathbf{p^*} \| \mathbf{x}(t)) +  D_{\mathrm{KL}}(\mathbf{q^*} \| \mathbf{y}(t))\big)}{dt} |_{t=t'}> 0.$ 
\end{cor}

Using this property as well as the fact that the replicator equations are diffemorphic to a system that preserves Lebesgue measure it is possible to show that the dynamics of zero-sum games with interior Nash equilibrium are Poincar\'{e} recurrent  \cite{piliouras2014optimization}.

\subsubsection{Games without interior Nash equilibria}
In this section, we consider in more detail games without interior Nash equilibria, such that a full picture of all possible dynamics can be given also in the stochastic case later on;
the analysis in this section follows mostly from \cite{piliouras2014optimization}.
The set of  equilibria in  
zero-sum games is convex. 
Such games exhibit a unique maximal support of Nash equilibrium strategies,\footnote{Indeed, if there exist two equilibria of maximal support with distinct supports, then the mixed strategy profile where each agent chooses uniformly at random to follow his randomized strategy in one of the two distributions is also a Nash equilibrium and has larger support than each of the original distributions.}
whose corresponding index sets we denote by $I$ and $J$; hence, the equilibrium is interior with respect to the simplex of this support.  We introduce
\begin{equation} \label{eq:SxF}
S_{F}^x \subset \partial \mathcal D  := \begin{cases}
    \{ x_i =0, \, \forall i \in F \}, \quad & \text{if } \ \emptyset \neq F \subsetneq [n] \\
    \partial \mathcal D, \quad & \text{otherwise},
\end{cases} 
\end{equation}
and
\begin{equation} \label{eq:SyG}
S_{G}^y \subset \partial \mathcal D := \begin{cases}
    \{ y_j =0, \, \forall j \in G \}, \quad & \text{if } \ \emptyset \neq G \subsetneq [m] \\
    \partial \mathcal D, \quad & \text{otherwise},
\end{cases} 
\end{equation}
and set
\begin{equation} \label{eq:Delta1}
\Delta_{\partial,1} := S_{I^c}^x \cap S_{J^c}^y,
\end{equation}
and
\begin{equation} \label{eq:Delta2}
\Delta_{\partial,2} := S_{I}^x \cap S_{J}^y,
\end{equation}
where $I^c, J^c$ denote the complements of $I,J$ respectively. For an anti-equilibrium, we define $\Delta_{\partial,1}^*$ and $\Delta_{\partial,2}^*$ accordingly.

\begin{thm}[\cite{piliouras2014optimization}]
\label{thm:notfullymixed}
If the replicator flow $\Phi$ does not have an interior equilibrium, then
given any interior starting point $x\in \mathcal D$,
the orbit $\Phi(x,\cdot)$ converges to the boundary.  
Furthermore, if $(\textbf p, \textbf q)$ is an equilibrium of maximum support corresponding with $\Delta_{\partial,1} \subset \partial \mathcal D$,
then its limit set $\omega(x)$ satisfies  $\omega(x)\subset \text{int}(\Delta_{\partial,1})$. 
\end{thm}

\begin{proof}
    See Appendix~\ref{app:determinstic_system}.
\end{proof}

\begin{cor}
Consider a generic zero-sum game with no interior Nash equilibrium. If  $(\textbf p^*, \textbf q^*)$ is an anti-equilibrium of maximum support,  corresponding with  $\Delta^*_{\partial,1} \subset \partial \mathcal D$, then 
$\lim_{t \rightarrow \infty}\big(D_{\mathrm{KL}}(\mathbf{p^*} \| \mathbf{x}(t)) +  D_{\mathrm{KL}}(\mathbf{q^*} \| \mathbf{y}(t))\big)=  \lim_{t \rightarrow \infty}  V\big((\mathbf{p^*},\mathbf{q^*});(\mathbf{x}(t),\mathbf{y}(t)) \big) = +\infty.$
\end{cor}
\begin{proof}
By Lemma~\ref{lem:supports} and the genericity of the zero-sum game we know that the maximal support index sets $I, J$ of equilibria and $I^*, J^*$ of anti-equilibria are not strictly contained in each other. By Theorem \ref{thm:notfullymixed} we know that the probabilities for all strategies not lying in $I,J$ converge to zero along trajectories, whereas the probabilities of all strategies in  $I,J$ stay bounded away from zero. Hence, there exists at least one summand in   
\begin{align*}
V^*(t)&:=V\big((\mathbf{p^*},\mathbf{q^*});(\mathbf{x}(t),\mathbf{y}(t)) \big) \\
&=  \sum_i (-p^*_i)\ln x_i(t) + \sum_j (-q^*_i)\ln y_j(t),
\end{align*}  
that diverges to infinity. Since all summands are non-negative , the quantity  $V^*(t)$ diverges, in fact monotonically, to infinity.
\end{proof}

\begin{ex}  \label{ex:deterministic}
In order to exemplify the distinction between dynamics for zero-sum games with interior Nash equilibrium and non-interior Nash equilibrium, as expressed in Theorem~\ref{thm:constant_of_motion}, we consider the following simple example of Matching Pennies (MP). The game is formalized as $2\times 2$-problem with payoff matrix
$
\mathbf{A} = 
\begin{pmatrix}
1 & -1 \\
-1 &  1
\end{pmatrix}
$
such that the dynamics, written in integral form for convenience of the stochastic perturbations to follow, are given by
\begin{equation} \label{eq:MP_2times2_determ}
\begin{array}{rl}
\rmd X_1(t) &= 2 X_1(t)(1- X_1(t))(2 Y_1(t)-1) \,
\rmd t,\\
\rmd Y_1(t) &= 2 Y_1(t)(1-Y_1(t))(1 - 2 X_1(t)) \,  \rmd t\,,
\end{array}
\end{equation}
and $X_2 = 1 - X_1$, $Y_2 = 1 -Y_1$. There is an interior Nash equilibrium, given by
$\mathbf p =\left( \frac{1}{2}, \frac{1}{2}\right), 
\mathbf q = \left( \frac{1}{2}, \frac{1}{2}\right)$.
We can modify the dynamics slightly by considering the problem with some additional strategy
$\mathbf{A} = 
\begin{pmatrix}
1 & -1 \\
-1 &  1 \\
-2 &  -2
\end{pmatrix}$
%
such that the replicator dynamics are
\begin{equation} \label{eq:MP_3times2_determ}
\begin{array}{rl}
\rmd X_1(t) &= X_1(t)[(1- X_1(t))(2 Y_1(t)-1) \\
&+ X_2(t)(2 Y_1(t)-1) + 2(1 - X_1(t) - X_2(t)) ]
\rmd t, \\
\rmd X_2(t) &= X_2(t)[(1- X_2(t))(1 -2 Y_1(t)) \\
&+ X_1(t) (1-2 Y_1(t)) + 2(1 - X_1(t) - X_2(t))]
\rmd t, \\
\rmd Y_1(t) &= 2Y_1(t)(1-Y_1(t))(X_2(t) - X_1(t)) \,  \rmd t\,,
\end{array}
\end{equation}
and $X_3 = 1 - X_1 - X_2$ and $Y_2 = 1 - Y_1$.
Now there is a non-interior Nash equilibrium of maximal support, given by
$
\mathbf p =\left( \frac{1}{2}, \frac{1}{2}, 0\right), \mathbf q = \left( \frac{1}{2}, \frac{1}{2}\right).
$
The corresponding anti-equilibrium is given by
$
\mathbf p^* =\left( 0, 0, 1\right), \mathbf q^* = \left( \frac{1}{2}, \frac{1}{2}\right).
$
Figure~\ref{fig:MP_determ} illustrates the dynamics of equations~\eqref{eq:MP_2times2_determ} and~\eqref{eq:MP_3times2_determ} respectively. As stated in Theorem~\ref{thm:constant_of_motion}, we observe trajectories coinciding with level sets of the cross entropy, taking the role of a constant of motion, for the situation with interior Nash equilibrium, and convergence to such levels sets within a subset of the boundary, corresponding with the Nash equilibrium of maximal support, in the case of no interior equilibrium.
\begin{figure}[htbp]
\centering
\begin{subfigure}[b]{0.4\textwidth}
\includegraphics[width=0.9\textwidth]{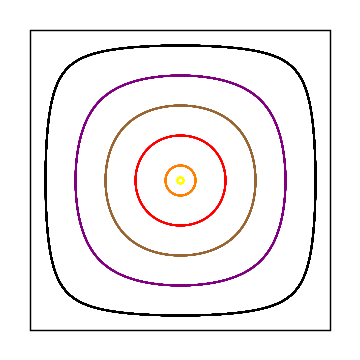} 
\caption{Sample orbits for Matching Pennies~\eqref{eq:MP_2times2_determ}} 
\end{subfigure}
\begin{subfigure}[b]{0.4\textwidth}
\includegraphics[width=0.9\textwidth]{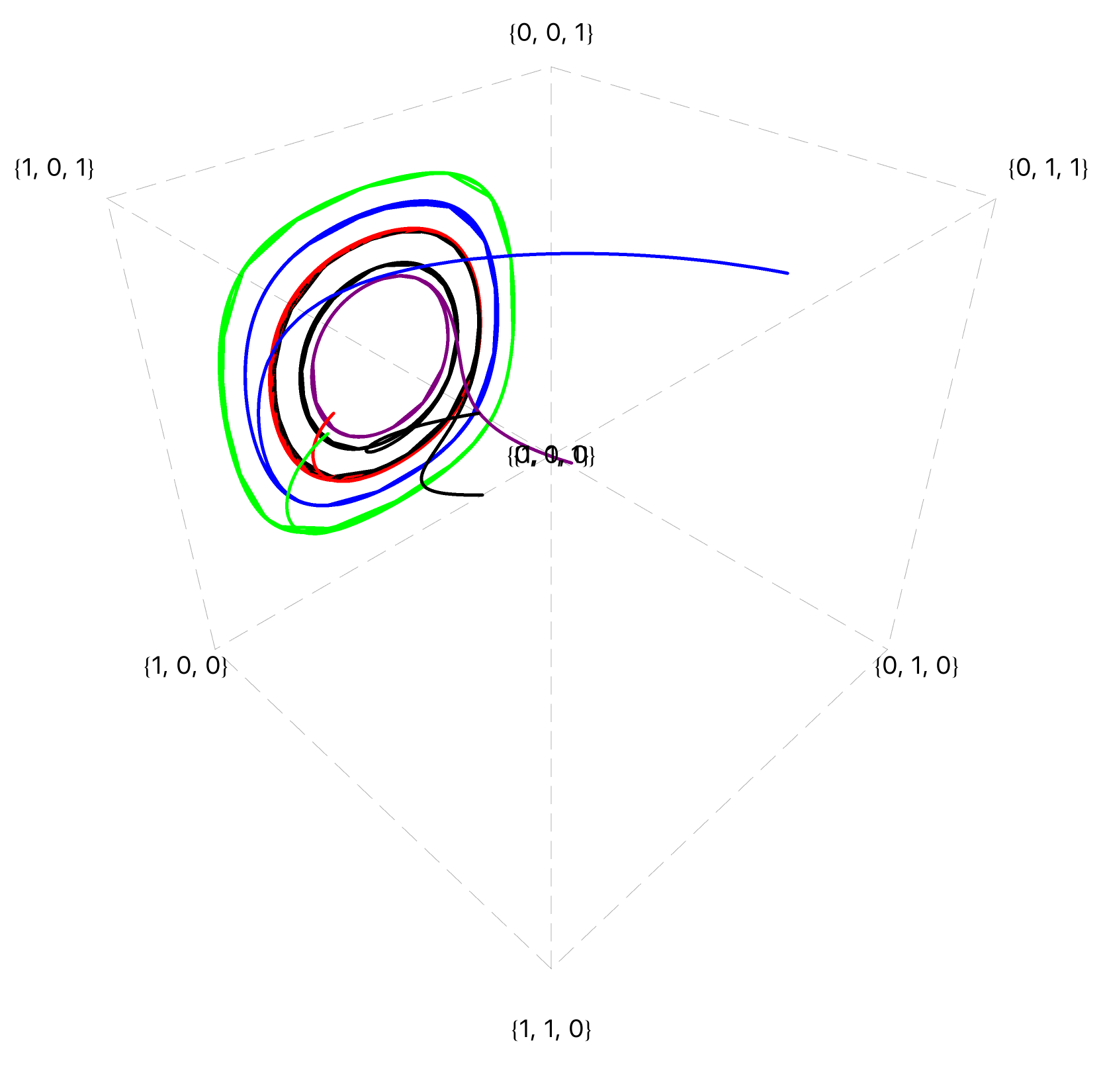}
  \caption{Sample orbits for $3\times2$ game~\eqref{eq:MP_3times2_determ}}
\end{subfigure}
\caption{Conservative dynamics with trajectories staying on levels sets of the cross entropy $V$~\eqref{eq:entropy} for replicator system with interior Nash equilibrium (a), and convergence to boundary corresponding with non-interior Nash equilibrium of maximum support for replicator system with entropy decrease (b).}
\label{fig:MP_determ}
\end{figure}
\end{ex}


\section{SDE model for stochastic replicator dynamics}
\label{sec:SDE_dynamics}
We now assume that the payoffs are exposed to random perturbations by independent Gaussian white noises.

The general model for working with stochastic perturbations of model~\eqref{eq:replicator2}, is the stochastic differential equation (SDE), in It\^{o} form,
\begin{equation} \label{eq:stoc2_general}
\begin{array}{rl}
\rmd X_i(t) &= X_i(t) \left(\{\mathbf{AY(t)}\}_i - \mathbf{X(t)^{\top} A Y(t)}\right) \rmd t \\
&+ X_i(t) (R(\mathbf{X})  \rmd W(t))_i\,,\\
\rmd Y_j(t) &= Y_j(t) \left(\{\mathbf{BX(t)}\}_j - \mathbf{Y(t)^{\top} B X(t)}\right) \rmd t \\
&+ Y_j(t) (S(\mathbf Y)  \rmd \tilde W(t))_j\,,
\end{array}
\end{equation}
where $W=(W_1, \dots, W_n)^{\top}$ and $\tilde W=(\tilde W_1, \dots, \tilde W_m)^{\top}$ are independent $n$-dimensional and $m$-dimensional Brownian motions, $(\mathbf X(0), \mathbf Y(0)) \sim \mu_0$ in $\overline{\mathcal D}$, where $\mu_0$ is some probability measure on $\overline{\mathcal D}$, and we have for all $(\mathbf X, \mathbf Y) \in \overline{\mathcal D}$
\begin{equation} \label{eq:matrix_property}
\mathbf{X}^{\top} R(\mathbf{X}) =0, \quad \mathbf{Y}^{\top} S(\mathbf{Y}) =0,
\end{equation}
where $R: \overline{\mathcal{D}} \to \mathbb{R}^{n\times n}$ and $S:\overline{\mathcal{D}} \to \mathbb{R}^{m\times m}$ are locally Lipschitz continuous.
In matrix form this equation reads
\begin{equation}  \label{eq:stoc2_matrix_general}
\hspace*{-1cm}
\begin{array}{ll}
&\rmd \mathbf{X(t)} =  \left(\diag (X_1(t), \dots, X_n(t))  - \mathbf{X(t)X(t)^{\top}}\right)\\ &\mathbf{A Y(t)} \, \rmd t 
+  \diag (X_1(t), \dots, X_n(t)) R(\mathbf{X})  \,  \rmd W(t),\\
\hspace*{-0.2cm}&\rmd \mathbf{Y(t)} = \left( \diag (Y_1(t), \dots, Y_m(t))  - \mathbf{Y(t)Y(t)^{\top}}\right)\\ &\mathbf{B X(t)} \, \rmd t 
 + \diag (Y_1(t), \dots, Y_m(t)) S(\mathbf Y) \, \rmd \tilde W(t).
\end{array}
\end{equation}
Note that property~\eqref{eq:matrix_property} means that
$ \sum_{i=1}^n X_i R_{i,j} = 0, \quad \forall j= 1, \dots,n,$
which implies for the noise term in equation~\eqref{eq:stoc2_general} that
$\sum_{i=1}^n X_i(t) (R(\mathbf{X})  \rmd W(t))_i =  \sum_{j=1}^n \left(\sum_{i=1}^n X_i(t) R_{i,j}\right) \rmd W(t)_j = 0.$
The same holds for $\mathbf Y$. Hence, we have $\sum_i \rmd X_i(t) =0$ and $\sum_j \rmd Y_j(t) =0$, and it is therefore guaranteed that $\overline{\mathcal{D}}$ is invariant with respect to equation~\eqref{eq:stoc2_general}. Hence, from the Lipschitz continuity of the drift and diffusion coefficients, we can infer that for any initial conditions $(\mathbf X(0), \mathbf Y(0)) \in \overline{\mathcal D}$, equation~\eqref{eq:stoc2_general} (and therefore also~\eqref{eq:stoc2_matrix_general}) has a unique strong solution.
 From the prefactors $X_i(t), Y_j(t)$ it is also immediate that $\partial \mathcal{D}$ and $\interior{\mathcal{D}}$ are invariant under the dynamics respectively.

We will also consider the following model, with a specific choice of the matrices $R$ and $S$, which we take as most natural for incorporating perturbations of model~\eqref{eq:replicator2}:
\begin{equation}  \label{eq:stoc2}
\begin{array}{rl}
\rmd X_i(t) &= X_i(t) \left(\{\mathbf{AY(t)}\}_i - \mathbf{X(t)^{\top} A Y(t)}\right) \rmd t \\
 &+ X_i(t) \big(-\sigma_1 X_1(t), \dots, \sigma_i(1-X_i(t)),\\
 & \dots, -\sigma_n X_n(t)\big)  \rmd W(t)\,,\\
\rmd Y_i(t) &= Y_i(t) \left(\{\mathbf{BX(t)}\}_i - \mathbf{Y(t)^{\top} B X(t)}\right) \rmd t \\
 &+ Y_i(t) \big(-\eta_1 Y_1(t), \dots, \eta_i(1-Y_i(t)),\\
 & \dots, -\eta_n Y_m(t)\big)  \rmd \tilde W(t)\,,
\end{array}
\end{equation}
with noise intensities $\sigma_1, \dots, \sigma_n > 0$ and $\eta_1, \dots, \eta_m > 0$ respectively.
In matrix form this equation reads
\begin{equation}  \label{eq:stoc2_matrix}
\begin{array}{ll}
\rmd \mathbf{X(t)} = &\left( \diag (X_1(t), \dots, X_n(t))  - \mathbf{X(t)X(t)^{\top}}\right)\\ &\big(\mathbf{A Y(t)} \, \rmd t +  \diag(\sigma_1, \dots, \sigma_n) \,  \rmd W(t) \big),\\
\rmd \mathbf{Y(t)} = &\left( \diag (Y_1(t), \dots, Y_m(t))  - \mathbf{Y(t)Y(t)^{\top}}\right)\\ &\big(\mathbf{B X(t)} \, \rmd t + \diag(\eta_1, \dots, \eta_m) \, \rmd \tilde W(t)\big).
\end{array}
\end{equation}
It is evident from Equation~\eqref{eq:stoc2_matrix} that our model describes uncertainty about the outcome of the game, at each time $t \geq0$, in terms of random fluctuations around the deterministic utility given by $\mathbf{A Y(t)}$ and $\mathbf{B X(t)}$ respectively.
It is easy to check that this is a particular version of our more general model, i.e.~the diffusion matrices in equation~\eqref{eq:stoc2} (and thereby~\eqref{eq:stoc2_matrix}) satisfy condition~\eqref{eq:matrix_property}.

\subsection{Some basic notions on Markov processes}
\label{sec:Markov_processes}
For the SDE~\eqref{eq:stoc2_general} on $\overline{\mathcal D} =\mathcal D \cup \partial \mathcal D$ and an initial condition $Z_0 =(\mathbf X_0, \mathbf Y_0) \in \overline{\mathcal D}$, we consider the time-homogeneous Markov process $(Z_t)_{t \geq 0}$ as solution to~\eqref{eq:stoc2_general}. Let $\mathcal B := \mathcal B \left(\overline{\mathcal D}\right)$ be the Borel $\sigma$-algebra and $\mathcal P$ the set of all probability measures with respect to $\mathcal B$. Then the process $(Z_t)_{t \geq 0}$ is associated with a family of probabilities $( \mathbb{P}_z)_{z \in \overline{\mathcal D}}$ on the filtered Wiener space $(\Omega, \mathcal{F}, ( \mathcal{F}_t)_{t \geq 0}, \mathbb P)$. We have
$$ \mathbb{P}_z(Z_0 = z) =1 \fa z \in \overline{\mathcal D}\,,$$
and the transition probabilities $(\hat P_t)_{t\geq 0}$ are given by
$$
  \hat P_t(z,A) = \mathbb{P}_z (Z_t \in A) \fa x \in \overline{\mathcal D} \mbox { and } A \in \mathcal B\,.
$$
The process is further associated with a semi-group of operators $(P_t)_{t\geq 0}$ given by
$$ P_t g(z) = \mathbb{E}_z [g(Z_t)]$$
for all measurable and bounded functions $g : \overline{\mathcal D}\to\R$. 
We say that a set $A \subset \overline{\mathcal D}$ is \emph{accessible} from a set $B \subset \overline{\mathcal D}$ if for every neighbourhood $U$ of $A$ and every $z \in B$, there is a $t\geq 0$ such that $P_t \mathds 1_U(z) > 0$, where $\mathds 1_U$ denotes the indicator function for the set $U$.

Writing $\mu (f) = \int f \rmd \mu$ for any $\mu \in \mathcal P$, we can introduce by duality to $P_t$ the semigroup
$$(P_t^* \mu)(f) = \mu (P_t f).$$
The set of \emph{invariant} measures (sometimes also called \emph{stationary} measures for the Markov process) is then given by
$$ \mathcal{P}_{\text{inv}} := \{ \mu \in \mathcal P \, : \, P_t^* \mu = \mu \}.$$
The set of \emph{ergodic} measures $\mathcal{P}_{\text{erg}} \subset \mathcal{P}_{\text{inv}}$ is given by the invariant measures $\mu$ such that any $P_t$-invariant measurable and bounded function $f$ is $\mu$-almost surely constant. As an application of Birhoff's Ergodic Theorem, one obtains that for any $\mu \in \mathcal{P}_{\text{erg}}$ we have for $\mu$-almost all $z$
$$ \lim_{t \to \infty} \frac{1}{t} \int_0^t P_t f(z) \rmd s = \int_{\overline{\mathcal D}} f(w) \, \rmd \mu(w)$$
and, by the usual construction of the corresponding dynamical system of the shift on a sequence space, even 
$$ \lim_{t \to \infty} \frac{1}{t} \int_0^t f(Z_s) \rmd s = \int_{\overline{\mathcal D}} f(y) \, \rmd \mu(y)$$
for $\mu$-almost all $z=Z_0$. 
We say that an invariant measure $\mu$ is a \emph{physical} measure if the above property holds for Lebesgue-almost all $z=Z_0$.

Consider the \emph{generator} or \emph{backward Kolmogorov operator} $\mathcal L$ for the Markov process given as
$$ \mathcal L f(z) = \lim_{t\downarrow 0} \frac{1}{t}\left(\mathbb{E}_z[f(Z_t)] - f(z)\right)$$
for measurable and bounded functions $f: \overline{\mathcal D} \to \R$.
Let us denote the ergodic measures for $(P_t)_{t\geq 0}$ which are supported on $\partial \mathcal D$ by $\mathcal P_{\textnormal{erg}}(\partial \mathcal D)$. Consider such an ergodic measure $\mu \in \mathcal P_{\textnormal{erg}}(\partial \mathcal D)$ and assume there is a function $V$ on $\mathcal D$ and a neighbourhood $ \supp{\mu} \subset U \subset \partial \mathcal D$ such that 
$\lim_{z \to  U} V(z) \to \infty.$
Setting $H = \mathcal L V$, we define the $H$-exponent with respect to $\mu$ as $\Lambda_{\mu}(H) = -\mu(H)$. If $\Lambda_{\mu}(H) < 0$ and $\supp(\mu)$ is accessible from $\mathcal D$, the measure $\mu$ is called \emph{attracting} with respect to $\mathcal D$ and if $\Lambda_{\mu}(H) > 0$, the measure $\mu$ is called \emph{repelling} with respect to $\mathcal{D}$.
Additionally, we define the $H$-exponents of the whole process as
\begin{align*}
\Lambda^+(H) &= - \inf_{ \mu \in \mathcal P_{\textnormal{erg}}(\partial \mathcal D) } \mu(H), \\ \Lambda^-(H) &= - \sup_{ \mu \in \mathcal P_{\textnormal{erg}}(\partial \mathcal D) } \mu(H).
\end{align*}
The process is called $H$-\emph{persistent} if $\Lambda^-(H) >0$ and $H$-\emph{nonpersistent} if $\Lambda^+(H) < 0$.

%


\subsection{The generator of the Markov process}
Recall from Theorem~\ref{thm:constant_of_motion} that in the deterministic case --- $R=S=0$ in~\eqref{eq:stoc2_general}, e.g.~$\sigma_j = \eta_j =0$ in model~\eqref{eq:stoc2} --- a crucial role is played by the entropy function
\begin{equation} \label{eq:V}
V(\mathbf{x, y}) = -\sum_i p_i\ln x_i - \sum_j q_i\ln y_j, 
\end{equation}
where $(\mathbf{p, q})$ is a Nash equilibrium for equation~\eqref{eq:replicator2}.

In order to understand the statistics of the Markov process solving equation~\eqref{eq:stoc2_general}, we study in more detail its generator $\mathcal L$. The operator $\mathcal L$ acts on $C^2$ functions $h: \overline{\mathcal{D}} \to \mathbb{R}$ as
\begin{align*}
\mathcal L h(\mathbf{x,y}) &= \sum_{i} x_i \left(\{\mathbf{Ay}\}_i - \mathbf{x^{\top} A y}\right) \partial_{x_i} h(\mathbf{x,y}) \\
&+ \sum_{i} y_i \left(\{\mathbf{Bx}\}_i - \mathbf{y^{\top} B x}\right) \partial_{y_i} h(\mathbf{x,y}) \\
 &+ \frac{1}{2}\sum_{i,j} D_{ij}(\mathbf{x}) \partial_{x_i x_j} h(\mathbf{x,y}) \\
 &+ \frac{1}{2}\sum_{i,j} \tilde D_{ij}(\mathbf{y}) \partial_{y_i y_j} h(\mathbf{x,y}), 
\end{align*}
where the diffusion matrices are given as
\begin{align*}
D_{ij}(\mathbf{x}) &= \sum_{k=1}^n x_i x_j R_{ik}(\mathbf{x})R_{jk}(\mathbf{x})\,, \\
\tilde D_{ij}(\mathbf{y}) &= \sum_{k=1}^m y_i y_j  S_{ik}(\mathbf{y})  S_{jk}(\mathbf{y})\,.
\end{align*}
We observe that
\begin{align*}
\partial_{x_i} V(\mathbf{x,y}) &= -\frac{p_i}{x_i}, \ \partial_{y_i} V(\mathbf{x,y}) = -\frac{q_i}{y_i}, \\
\partial_{x_j x_i} V(\mathbf{x,y}) &= \delta_{i,j} \frac{-p_i}{x_i^2}, \ \partial_{y_j y_i} V(\mathbf{x,y}) = \delta_{i,j} \frac{-q_i}{y_i^2}.
\end{align*}
Therefore, applying $\mathcal L$ to $V(\mathbf{x,y})$, we obtain 
\begin{align} \label{eq:H_bKolm_gen}
&\mathcal L V(\mathbf{x,y}) = -\sum_{i} p_i \left(\{\mathbf{Ay}\}_i - \mathbf{x^{\top} A y}\right) \nonumber\\
&- \sum_{i} q_i \left(\{\mathbf{Bx}\}_i - \mathbf{y^{\top} B x}\right) + \sum_{i}  p_i \left( \sum_{k =1}^n R_{ik}^2(\mathbf{x}) \right) \nonumber\\
 &+ \sum_{i}  q_i \left( \sum_{k=1}^n S_{ik}^2(\mathbf{y}) \right).
\end{align}
In the more specific situation of model~\eqref{eq:stoc2}, the diffusion coefficients can be written as
\begin{align*}
D_{ij}(\mathbf{x}) &= \sum_{k=1}^n G_{ik}(\mathbf{x})G_{jk}(\mathbf{x})\,, \\
\tilde D_{ij}(\mathbf{y}) &= \sum_{k=1}^m \tilde G_{ik}(\mathbf{y}) \tilde G_{jk}(\mathbf{y})\,,
\end{align*}
for, as we can see directly from equation~\eqref{eq:stoc2_matrix},
\begin{align*}
G(\mathbf{x}) &= \left[ \diag (x_1, \dots,x_n)  - \mathbf{x x^{\top}}\right] \diag(\sigma_1, \dots, \sigma_n), \\
\tilde G (\mathbf{y}) &= \left[ \diag (y_1, \dots, y_m)  - \mathbf{y y^{\top}}\right] \diag(\eta_1, \dots, \eta_m).
\end{align*}
In particular, we can write the explicit coefficients
\begin{align*}
G_{ik}(\mathbf{x}) &= \begin{cases}
\sigma_k x_i x_k & \text{ if } i\neq k,\\
\sigma_i x_i(1-x_i) & \text{ if } i = k,
\end{cases} \\
\tilde G_{ik}(\mathbf{y}) &= \begin{cases}
\eta_k y_i y_k & \text{ if } i\neq k,\\
\eta_i y_i(1-y_i) & \text{ if } i = k.
\end{cases}
\end{align*}
Hence, in this context, equation~\eqref{eq:H_bKolm_gen} can be written as
\begin{align} \label{eq:H_bKolm}
&\mathcal L V(\mathbf{x,y}) = -\sum_{i} p_i \left(\{\mathbf{Ay}\}_i - \mathbf{x^{\top} A y}\right) - \sum_{i} q_i \big(\{\mathbf{Bx}\}_i \nonumber\\
&- \mathbf{y^{\top} B x}\big) 
 + \sum_{i}  (p_i)\left( \sum_{k \neq i} \frac{\sigma_k^2}{2}x_k^2 + \frac{\sigma_i^2}{2}(1-x_i)^2\right)\nonumber\\
 &+ \sum_{i}  (q_i)\left( \sum_{k \neq i} \frac{\eta_k^2}{2}y_k^2 + \frac{\eta_i^2}{2}(1-y_i)^2\right).
\end{align}
%

For detailed computations 
see Appendix~\ref{App:second_derivatives}.

One observes from the proof of Theorem~\ref{thm:constant_of_motion} (see Appendix~\ref{app:determinstic_system}) that, for any zero-sum game with payoff matrices $\mathbf A, \mathbf B$ and  Nash equilibrium $(\mathbf p,\mathbf q)$, we have for any strategy profile $(\mathbf x, \mathbf y)$ that
\begin{equation}\label{eq:utilities_inequ}
\sum_{i} p_i \left(\{\mathbf{Ay}\}_i - \mathbf{x^{\top} A y}\right) + \sum_{j} q_j \left(\{\mathbf{Bx}\}_j - \mathbf{y^{\top} B x}\right) \geq 0 
\end{equation}   
Note that the diffusive part in \eqref{eq:H_bKolm} is only contributing positively which will push trajectories outside in this case, as opposed to, e.g., model \eqref{eq:Imhof_model}, whose derivation involves an It\^{o} correction that leads to an additional term acting as dissipation towards the interior.
If the Nash equilibrium $(\mathbf p,\mathbf q)$ is interior, inequality~\eqref{eq:utilities_inequ} becomes an equality.
In the case where no interior Nash equilibrium exists, an even stronger statement is possible: there is an equilibrium of maximal support, which is the barycenter of the equilibrium polytope, such that for any interior strategy vector $(x,y)$ \cite{GeorgiosSODA18} 
\begin{equation}
\label{eq:equil}
    \sum_{i} p_i \left(\{\mathbf{Ay}\}_i - \mathbf{x^{\top} A y}\right) + \sum_{j} q_j \left(\{\mathbf{Bx}\}_j - \mathbf{y^{\top} B x}\right) > 0.
    \end{equation}
Formula~\eqref{eq:utilities_inequ} actually generalizes to network extensions of zero-sum games with arbitrary many agents \cite{GeorgiosSODA18, piliouras2014optimization}. 
Analogously, we obtain for a non-interior anti-equilibrium $(\mathbf p^*, \mathbf q^*)$
\begin{equation} \label{eq:antiequ}
 \sum_{i} p^*_i \left(\{\mathbf{Ay}\}_i - \mathbf{x^{\top} A y}\right) + \sum_{j} q^*_j \left(\{\mathbf{Bx}\}_j - \mathbf{y^{\top} B x}\right)  <0.
\end{equation}

\subsection{Locating invariant measures of the Markov process}
\label{sec:invariant_measure}
Recall that we consider the stochastic model~\eqref{eq:stoc2_general}, and in particular model~\eqref{eq:stoc2}, on the domain
$\mathcal{D} := \Delta_n \times \Delta_m $, where
$$ \Delta_n := \{ \mathbf{x} \in (0,1)^n \,:\, x_1 + \dots + x_n =1  \}$$
and $\Delta_m$ is defined analogously. 
We aim to give a characterization of the dynamics by finding invariant measures for the Markov process solving the SDE~\eqref{eq:stoc2} and, in particular, the invariant measures whose supports attract trajectories in the sense of Section 4.1.
A first useful observation concerns the invariance under the stochastic dynamics of the interior and the boundary respectively:
\begin{lem} \label{lem:invariantsets}
The sets $\mathcal{D}$ and $\partial \mathcal{D}$ respectively are invariant under the process $Z_t =(X_t, Y_t)$ solving equation~\eqref{eq:stoc2}.
In particular, any invariant probability measure $\mu$ for $Z_t$ on $\overline{\mathcal D}= \mathcal{D} \cup \partial \mathcal{D}$ can be written as a sum 
$$\mu = \alpha \mu_{\mathcal{D}} + (1- \alpha)\mu_{\partial \mathcal{D}}, \ \alpha \in [0,1], $$ 
where $\mu_{\mathcal{D}}$ and $\mu_{\partial \mathcal{D}}$ are invariant probability measures on $\mathcal{D}$ and $\partial \mathcal{D}$ respectively.
\end{lem}
\begin{proof}
Follows directly from the form of equation~\eqref{eq:stoc2_general} and condition~\eqref{eq:matrix_property}.
\end{proof}
The following proof of our main theorem uses an abstract persistence result \cite[Theorem 5.4]{BenaimStrickler} that fits the setting of a Feller semigroup $(P_t)_{t \geq 0}$ as introduced in Section~\ref{sec:Markov_processes} for the SDE~\eqref{eq:stoc2_general}. For that, we choose $\partial \mathcal{D} = \mathcal X_0$ and $\mathcal{D} = \mathcal X_+$ in the context of \cite[Section 5]{BenaimStrickler}.
We note that \cite[Hypothesis 5.1]{BenaimStrickler}, which states the invariance of $\mathcal X_0$ ($\partial \mathcal{D}$ for our case) under $P_t$, is obviously satisfied for the Feller semigroup $(P_t)_{t\geq0}$ induced by the SDE solution $(X_t)_{t \geq 0}$. 
Furthermore, we can summarize \cite[Hypothesis 5.2]{BenaimStrickler} and \cite[Theorem 5.4]{BenaimStrickler} into the following statement, introducing for $f$ and $f^2$ lying in the domain of the generator $\mathcal L$ the operator
$\Gamma(f) = \mathcal L f^2 - 2 f \mathcal L f$ and recalling the definitions of $H$-(non)persistence and accessibility from the end of Section 4.1:
\begin{thm}[\cite{BenaimStrickler}] \label{thm:prep}
Assume that the there are continuous maps $V: \mathcal{D} \to \R^+$ and $H:\overline{\mathcal D} \to \R$ with the following properties:
\begin{enumerate}[(a)]
    \item For all compact $K \subset \mathcal D$, there exists $V_K$ with $V|_K = V_K|_K$ and $(\mathcal L V_K)|_K= H|_K$;
    \item $ \sup_{\{K\subset \mathcal D : K \text{ compact}\}} \|\Gamma(V_K)|_K\| < \infty$;
    \item $\lim_{x \to \partial \mathcal D} V(x) = \infty$;
    \item There is a $\kappa > 0$ such that $\left| V(X_t) - V(X_{t-})\right| \leq \kappa$, i.e.~jumps of $V(X_t)$ are bounded;
\end{enumerate}
Then, if the process is $H$-nonpersistent and $\partial \mathcal D$ is accessible from $\overline{\mathcal D}$, we have
\begin{equation*}
\mathbb{P}_x \left( \liminf_{t \to \infty} \frac{V(X_t)}{t} \geq - \Lambda^+(H) \right) = 1,
\end{equation*}
for all $X_0 = x \in \overline{\mathcal D}$.
\end{thm}
In the following we will make one more ellipticity assumption on $R$ and $S$ which is, e.g., directly satisfied by model~\eqref{eq:stoc2_matrix}:
\begin{align} \label{eq:matrix_cond}
\exists \xi > 0 \text{ such that } \forall i\neq j:\, &\sum_{k=1}^n R_{ik}^2(x) + \sum_{k=1}^n R_{jk}^2(x) \geq \xi, \nonumber\\
 &\sum_{k=1}^n R_{ik}^2(x) = 0 \text{ iff } x_i =1.
\end{align}
We are now prepared to show our main result:


\begin{thm}[Zero-sum game with noise] \label{thm:general}
Consider model~\eqref{eq:stoc2_general} with the assumptions as above, in particular the matrices $R$ and $S$ satisfying~\eqref{eq:matrix_cond}. Then
\begin{enumerate}[(a)]
\item any invariant probability measure $\mu$ on $\overline{\mathcal{D}}$ is 
\begin{enumerate}[(i)]
\item supported on the boundary $ \partial \mathcal D$,
\item given by a convex combination of the ergodic Dirac measures
$\delta_{v_{i,j}}$, $(i,j) \in \{1, \dots ,n\} \times \{1, \dots ,m\} $, supported on the corners $v_{i,j}$ of $\partial \mathcal D$.
\end{enumerate}
\item If the Nash equilibrium $(\mathbf p, \mathbf q)$ is interior, all $\delta_{v_{i,j}}$ are attracting with respect to $\mathcal D$.
\item If there is no interior Nash equilibrium but only a Nash equilibrium $(\mathbf p, \mathbf q)$ with maximal support, then
\begin{enumerate}[(i)]
\item for sufficiently large noise, i.e.~$R(x)$ and $S(y)$ with sufficiently large entries,
all $\delta_{v_{i,j}}$ are attracting with respect to $\mathcal D$.
\item  otherwise, for sufficiently small noise, i.e.~$R(x)$ and $S(y)$ with sufficiently small entries,
the only invariant measures which are attracting with respect to $\mathcal D$ are contained in the subset $\Delta_{\partial, 1}$ of $ \partial \mathcal D$ which contains the Nash equilibrium of maximal support.
\end{enumerate}
\item In particular, for model~\eqref{eq:stoc2}, a large noise condition in the sense of (c)(i) reads
\begin{align} \label{eq:cond_large_noise}
&\min_{i \in I,j \in J} \left\{ p_i  \min_{k} \left\{\frac{\sigma_k^2}{2n}\right\} + q_j  \min_{k} \left\{\frac{\eta_k^2}{2m}\right\} \right\} \nonumber \\
&> \sum_{i \in I} p_i\{\mathbf{ A q}\}_i + \max_{i \in I^c} \sum_{j \in J} q_j  B_{j,i}  + \max_{j \in J^c} \sum_{i \in I} p_i  A_{i,j},
\end{align}
and a small noise condition in the sense of (c)(ii) can be written as
\begin{align} \label{eq:cond_small_noise}
 \sum_{i \in I} p_i\{\mathbf{ A q}\}_i +  \min_{i \in I^c, \, j \in J^c } &\left\{ \sum_{k \in J} q_k B_{k,i}, \,  \sum_{l \in I} p_l  A_{l,j} \right\} \nonumber\\
&>\max_{i,j} \left\{ n \frac{\sigma_i^2}{2}+ m \frac{\eta_j^2}{2} \right\}.
\end{align}
%
\end{enumerate} 
\end{thm}

 In the following we will prove our statements by a case distinction in terms of having an interior Nash equilibrium or not. The latter case will require an additional case distinction depending on a comparison of the noise and deterministic terms.

\begin{proof}
As before, consider the cross entropy function
\begin{equation} \label{tildeH}
V(\mathbf{x, y}) := \sum_i (-p_i)\ln x_i + \sum_j (-q_i)\ln y_j.
\end{equation}
Then $V$ is smooth and positive on the domain 
$\mathcal{D} := \Delta_n \times \Delta_m $.

\textbf{1.} Firstly, assume that the Nash equilibrium $(\mathbf p,\mathbf q)$ is interior.
 Similarly to~\eqref{eq:H_bKolm}, we observe for $V$, given by equation~\eqref{tildeH}, that
\begin{align} \label{eq:LV}
H(\mathbf{x,y}) &:= \mathcal L V(\mathbf{x,y}) = \sum_{i} (-p_i) \left(\{\mathbf{Ay}\}_i - \mathbf{x^{\top} A y}\right) \nonumber\\
&+ \sum_{i} (-q_i) \left(\{\mathbf{Bx}\}_i - \mathbf{y^{\top} B x}\right) \nonumber\\
 &+ \sum_{i}  p_i \left( \sum_{k =1}^n R_{ik}^2(\mathbf{x}) \right)
 + \sum_{i}  q_i \left( \sum_{k=1}^n S_{ik}^2(\mathbf{y}) \right). \nonumber\\
&= \sum_{i}  p_i \left( \sum_{k =1}^n R_{ik}^2(\mathbf{x}) \right)
 + \sum_{i}  q_i \left( \sum_{k=1}^n S_{ik}^2(\mathbf{y}) \right),
\end{align}
where the last equality follows from the considerations around inequality~\eqref{eq:utilities_inequ}, since the Nash equilibrium is interior. The fact that $V$ satisfies assumptions (a) and (d) in Theorem~\ref{thm:prep} are obvious and assumption (b) follows from a straight-forward calculation. Property (c) can directly be observed as
\begin{equation} \label{eq:limV}
\lim_{z \to \partial \mathcal{D}} V(z) \to \infty.
\end{equation}
Using condition~\eqref{eq:matrix_cond}, we observe directly from~\eqref{eq:LV} that $H= \mathcal L V(\mathbf{x,y}) > 0$ on $\mathcal{D} \cup \partial \mathcal{D}$. 
Hence, $\Lambda^+(H) < 0$ such that the process is $H$-nonpersistent, and, in particular, claim (b) of our Theorem follows. 
As can be seen immediately from our ellipticity condition~\eqref{eq:matrix_cond}, $\partial \mathcal D$ is accessible from $\overline{\mathcal D}$ such that we can deduce with Theorem~\ref{thm:prep} that
\begin{equation*}
\mathbb{P}_x \left( \liminf_{t \to \infty} \frac{V(X_t)}{t} \geq - \Lambda^+(H) \right) = 1,
\end{equation*}
for all $X_0 = x \in \overline{\mathcal{D}}$.
In other words, the boundary $\partial \mathcal{D}$ absorbs almost all trajectories, implying $\alpha=0$ in Lemma~\ref{lem:invariantsets} and thereby the claim (a)(i) for this case.

 Additionally note that for model~\eqref{eq:stoc2}, the remaining term in~\eqref{eq:LV} reads
\begin{align*}
\sum_{i}  (p_i)\left( \sum_{k \neq i} \frac{\sigma_k^2}{2}x_k^2 + \frac{\sigma_i^2}{2}(1-x_i)^2\right)\nonumber\\
 + \sum_{i}  (q_i)\left( \sum_{k \neq i} \frac{\eta_k^2}{2}y_k^2 + \frac{\eta_i^2}{2}(1-y_i)^2\right),
\end{align*} 
such that we have on $\partial \mathcal D$  the uniform lower bound
\begin{align} \label{eq:lowerbound}
H \geq\min_{i,j} \bigg\{ p_i \left( \min_{k \neq i} \left\{\frac{\sigma_k^2}{2n}\right\}+ \frac{\sigma_i^2}{2} \right), \nonumber\\
q_j \left( \min_{k \neq j} \left\{\frac{\eta_k^2}{2m}\right\}+\frac{\eta_j^2 }{2} \right) \bigg\} > 0.
\end{align} 
 We will come back to this bound later in the proof.

\textbf{2.} We now consider the situation that there is no interior Nash equilibrium. Then there is a non-interior equilibrium $(\mathbf{p}, \mathbf{q})$ with maximal support, whose index sets we denote
by $I$ and $J$; hence, the equilibrium is interior with respect to the simplex of this support.

 Recall the sets $\Delta_{\partial, 1}$~\eqref{eq:Delta1} and  $\Delta_{\partial, 2}$~\eqref{eq:Delta2}.
 We now introduce new Lyapunov-type functions, taking into account the different subsets of the boundary, $\Delta_{\partial, 1}$ and $\Delta_{\partial, 2}$, accordingly.
Firstly, we take, similarly to before, the Lyapunov function
\begin{equation} \label{V0}
V_0(\mathbf{x, y}) := \sum_{i \in I} (-p_i)\ln x_i + \sum_{j \in J} (-q_j)\ln y_j,
\end{equation}
such that $V_0(\mathbf{x, y}) \to \infty$ when $(\mathbf{x, y}) \to \Delta_{\partial, 2}$ and $V_0(\mathbf{x, y}) < \infty$ when $(\mathbf{x, y}) \in \interior (\Delta_{\partial, 1})$.
We define for any vector $(\mathbf{x, y})  \in \overline{\mathcal D}$
\begin{align*}
L_0(\mathbf{x, y})  &:= \sum_{i \in I} (-p_i) \left(\{\mathbf{ A y}\}_i - \mathbf{x^{\top}  A y}\right)\\
&+ \sum_{j \in J} (-q_j) \left(\{\mathbf{ B x}\}_j - \mathbf{y^{\top}  B x}\right) \\
&= - \left(\sum_{i \in I} p_i\{\mathbf{ A y}\}_i + \sum_{j \in J} q_j \{\mathbf{ B x}\}_j\right).
\end{align*}
In particular, we have that
\begin{align} \label{eq:L0_expression}
L_0(\mathbf{x,y})
&= - \left(\sum_{i \in I} p_i\{\mathbf{ A y}\}_i + \sum_{j \in J} q_j \{\mathbf{ B x}\}_j \right) \nonumber\\
&=  - \bigg(\sum_{i \in I} p_i\{\mathbf{ A y}\}_i + \sum_{j \in J} q_j \{\mathbf{ B x}\}_j \nonumber\\
&- \sum_{j \in J, i \in I^c} q_j  B_{j,i} x_i  - \sum_{i \in I, j \in J^c} p_i  A_{i,j} y_j \nonumber\\
&+ \sum_{j \in J, i \in I^c} q_j  B_{j,i} x_i +  \sum_{i \in I, j \in J^c} p_i  A_{i,j} y_j\bigg) \nonumber \\
&= - \bigg( \sum_{i\in I} p_i\{\mathbf{ A q}\}_i (\sum_{j \in J}y_j- \sum_{i \in I}x_i) \nonumber\\
&+  \sum_{j \in J, i \in I^c} q_j  B_{j,i} x_i +  \sum_{i \in I, j \in J^c} p_i  A_{i,j} y_j\bigg).
\end{align}
Similarly to before, for $(\mathbf{x,y}) \in \overline{\mathcal D}$, we consider
\begin{equation} \label{eq:JK}
H_0(\mathbf{x,y}) := \mathcal{L} V_0 (\mathbf{x,y}) =  L_0(\mathbf{x,y}) + K_0(\mathbf{x,y}),
\end{equation}
where, in general,
$$ K_0(\mathbf{x,y}) = \sum_{i \in I}  p_i \left( \sum_{k =1}^n R_{ik}^2(\mathbf{x}) \right)
 + \sum_{j \in J}  q_j \left( \sum_{k=1}^n S_{jk}^2(\mathbf{y}) \right),$$
and specifically for model~\eqref{eq:stoc2}
\begin{align} \label{eq:K0_special}
K_0(\mathbf{x,y}) = \sum_{i \in I}  p_i \left( \sum_{k \neq i} \frac{\sigma_k^2}{2}x_k^2 + \frac{\sigma_i^2}{2}(1-x_i)^2\right) \nonumber \\
 + \sum_{j \in J}  q_j \left( \sum_{k \neq j} \frac{\eta_k^2}{2}y_k^2 + \frac{\eta_j^2}{2}(1-y_j)^2\right).
 \end{align}
Recall that there is also an anti-equilibrium of maximal support, denoted by $(\mathbf p^*, \mathbf q^*)$, which is also not interior by Lemma~\ref{lem:supports}.
We now additionally consider the function
\begin{equation} \label{V_noint}
V_1(\mathbf{x, y}) := - \sum_{i} p_i^* \ln x_i - \sum_{j} q_j^* \ln y_j,
\end{equation}
which we also see as a Lyapunov-type function on $\mathcal D$, going to infinity when approaching $\Delta_{\partial,1}$ due to Lemma~\ref{lem:supports}.
Writing,
\begin{align}\label{eq:L1}
L_1(\mathbf{x,y}) := &- \sum_{i}  p_i^* \left(\{\mathbf{ A y}\}_i - \mathbf{x^{\top}  A y}\right) \nonumber\\
& - \sum_{j} q_j^* \left(\{\mathbf{ B x}\}_j - \mathbf{y^{\top}  B x}\right)
\end{align}
and
$$ K_1(\mathbf{x,y}) = \sum_i  p_i^* \left( \sum_{k =1}^n R_{ik}^2(\mathbf{x}) \right)
 + \sum_j  q_j^* \left( \sum_{k=1}^n S_{jk}^2(\mathbf{y}) \right),$$
where for model~\eqref{eq:stoc2} this reads
\begin{align*}
K_1(\mathbf{x,y}) = &\sum_{i} p_i^* \left( \sum_{k \neq i} \frac{\sigma_k^2}{2}x_k^2 + \frac{\sigma_i^2}{2}(1-x_i)^2\right) \\
& + \sum_{j}  q_j^* \left( \sum_{k \neq j} \frac{\eta_k^2}{2}y_k^2 + \frac{\eta_j^2}{2}(1-y_j)^2\right),
\end{align*} 
we have from inequality~\eqref{eq:antiequ} that $H_1$ satisfies
\begin{equation} \label{eq:H1}
H_1(\mathbf{x,y}) = \mathcal{L} V_1(\mathbf{x,y}) = L_1(\mathbf{x,y}) + K_1(\mathbf{x,y}) > 0,
\end{equation}
for all $\mathbf x, \mathbf y \in \overline{\mathcal{D}}$ except if $(\mathbf p^*, \mathbf q^*)$ is a corner point and $(\mathbf{x,y}) = (\mathbf p^*, \mathbf q^*)$ (then $H_1(\mathbf{x,y}) =0$.)

\textbf{2.1.} Assume we have 
\begin{equation*}
  K_0(\mathbf{x,y}) >  - L_0(\mathbf{x,y}), \ \text{i.e. } H_0 (\mathbf{x,y}) > 0,
\end{equation*}
for all $(\mathbf{x,y}) \in \partial \mathcal D$,
except if $(\mathbf p, \mathbf q)$ is a corner point and $(\mathbf{x,y}) = (\mathbf p,\mathbf q)$ such that $H_0(\mathbf{x,y}) =0$ --- but then $H_1(\mathbf{x,y}) > 0$ in this case due to Lemma~\ref{lem:supports}. 
Now we can follow the same arguments as before, upon introducing for completeness and simplicity  the remainder index sets
\begin{align*}
\tilde I &= [n] \setminus \big( supp(\mathbf p) \cup supp(\mathbf p^*) \big), \\
 \tilde J &= [m] \setminus \big( supp(\mathbf q) \cup supp(\mathbf q^*) \big),
\end{align*}  
with
\begin{equation} \label{V_2}
V_2(\mathbf{x, y}) := - \sum_{i\in\tilde I} \ln x_i - \sum_{j \in \tilde J} \ln y_j.
\end{equation}
Similarly to before, we write
\begin{align*}
L_2(\mathbf{x,y}) := &- \sum_{i\in\tilde I}  \left(\{\mathbf{ A y}\}_i - \mathbf{x^{\top}  A y}\right) \\
& - \sum_{j\in\tilde J} \left(\{\mathbf{ B x}\}_j - \mathbf{y^{\top}  B x}\right)
\end{align*}
and
$$ K_2(\mathbf{x,y}) = \sum_{i \in \tilde I}  \left( \sum_{k =1}^n R_{ik}^2(\mathbf{x}) \right)
 + \sum_{j \in \tilde J}  \left( \sum_{k=1}^n S_{jk}^2(\mathbf{y}) \right),$$
such that we have $H_2$ on $\overline{\mathcal{D}}$ given by
\begin{equation} \label{eq:H2}
H_2(\mathbf{x,y}) = \mathcal{L} V_2(\mathbf{x,y}) = L_2(\mathbf{x,y}) + K_2(\mathbf{x,y}).
\end{equation}
 Note that, if $\tilde I$ and $\tilde J$ are both empty, i.e.~equilibrium and anti-equilibrium together have full support, we simply have $V_2 = H_2 = 0$.

For $\alpha, \beta, \gamma \in [0,1]$ such that $\alpha + \beta + \gamma =1$, we can summarize these quantities into
$V_{\alpha, \beta, \gamma} = \alpha V_0 + \beta V_1 + \gamma V_2$
and, by linearity,
$H_{\alpha, \beta, \gamma} =  \mathcal L V_{\alpha, \beta, \gamma} = \alpha H_0 + \beta H_1 + \gamma H_2.$
Now $V_{\alpha, \beta, \gamma}$ is a Lyapunov function for $\mathcal D$ and $\partial \mathcal D$ with the same properties as before, in particular $H_{\alpha, \beta, \gamma} > 0$ on $\partial \mathcal D$ for $\gamma > 0$ small enough.
Condition~\eqref{eq:cond_large_noise} for the situation of the specific model~\eqref{eq:stoc2} can be derived similarly to bound~\eqref{eq:lowerbound} and using equation~\eqref{eq:L0_expression}.

\textbf{2.2.} 
We now treat the case with sufficiently small noise in order to show concentration on the part of the boundary corresponding with the support of
$(\mathbf p, \mathbf q)$. 

Note from~\eqref{eq:equil} (and the proof of Theorem A.1.) that $ - L_0(\mathbf{x,y}) >0$ for all $(\mathbf x, \mathbf y) \notin \Delta_{\partial, 1}$. Assuming without loss of generality that $\sum_{i\in I} p_i\{\mathbf{ A q}\}_i =0$, we can deduce from~\eqref{eq:L0_expression} that
$$ - L_0(\mathbf{x,y}) = \sum_{j \in J, i \in I^c} q_j  B_{j,i} x_i +  \sum_{i \in I, j \in J^c} p_i  A_{i,j} y_j $$
Hence, we obtain that
$$\min_{i \in I^c} \sum_{j \in J} q_j  B_{j,i} > 0, \quad   \min_{j \in J^c} \sum_{i \in I} p_i  A_{i,j}  > 0,$$ 
if $\emptyset \neq I \subsetneq [n], \emptyset \neq J \subsetneq [m]$ which holds for at least one of them. Without loss of generality, we will assume this for both index sets in the following since, if one of them is empty or the full index set, the argument reduces immediately to the other index set in a straightforward way.

Thus, we can choose noise terms such that
\begin{align}
    \min_{i \in I^c} \sum_{j \in J} q_j  B_{j,i} &> \max_{y} \max_{j} \left\{ \sum_{k=1}^n S_{jk}^2(\mathbf{y}) \right\}  \label{eq:cond_small_noise_gen_y}\\
\min_{j \in J^c} \sum_{i \in I} p_i  A_{i,j} &>  \max_{x} \max_{i} \left\{ \sum_{k =1}^n R_{ik}^2(\mathbf{x})  \right\}. \label{eq:cond_small_noise_gen_x}
\end{align}
For the special case of model~\eqref{eq:stoc2}, we can easily derive, using~\eqref{eq:K0_special}, the sufficient conditions
\begin{align}
  \min_{i \in I^c} \sum_{j \in J} q_j  B_{j,i}   
&>\max_{j} \left\{ m \frac{\eta_j^2}{2} \right\}, \label{eq:cond_small_noise_spec_y} \\
\min_{j \in J^c} \sum_{i \in I} p_i  A_{i,j} 
&> \max_{i} \left\{ n \frac{\sigma_i^2}{2} \right\}, \label{eq:cond_small_noise_spec_x}
\end{align}
to satisfy assumptions~\eqref{eq:cond_small_noise_gen_y} and~\eqref{eq:cond_small_noise_gen_x}.
Conditions~\eqref{eq:cond_small_noise_spec_y} and~\eqref{eq:cond_small_noise_spec_x} are clearly implied by condition~\eqref{eq:cond_small_noise}, as we assumed $\sum_{i \in I} p_i\{\mathbf{ A q}\}_i = 0$ without loss of generality.

In particular, under conditions~\eqref{eq:cond_small_noise_gen_y} and~\eqref{eq:cond_small_noise_gen_x}, we have
\begin{align*}
     K_0(\mathbf{x,y}) &=  \sum_{i \in I}  p_i \left( \sum_{k =1}^n R_{ik}^2(\mathbf{x}) \right)
 + \sum_{j \in J}  q_j \left( \sum_{k=1}^n S_{jk}^2(\mathbf{y}) \right) \\
 &< - L_0(\mathbf{x,y}),
 \end{align*}
 i.e.~$H_0(\mathbf{x,y}) < 0$, for all $(\mathbf{x,y}) \in \Delta_{\partial,2}$. This implies that any invariant measures supported on $\Delta_{\partial,2}$ are repelling with respect to $\mathcal D$ (cf.~Section~4.1).

For any $(\mathbf{x,y}) \in \Delta_{\partial,1}$, we have
$ K_1(\mathbf{x,y})  > 0$ by condition~\eqref{eq:matrix_cond}, since $ K_1(\mathbf{x,y})  =0$ iff $(\mathbf{x,y}) = (\mathbf p^*, \mathbf q^*)$ is a corner point which then cannot be in $\Delta_{\partial,1}$.
Additionally, we know from inequality~\eqref{eq:antiequ} that $L_1(\mathbf{x,y}) \geq 0 $. Hence, we obtain $H_1(\mathbf{x,y}) > 0$ for all $(\mathbf{x,y}) \in \Delta_{\partial,1}$, and, even more generally that $H_1(\mathbf{x,y}) > 0$ for $(\mathbf{x,y}) \in \partial \mathcal D$  unless $(\mathbf{x,y}) = (\mathbf p^*, \mathbf q^*)$ is a corner point (cf.~also~\eqref{eq:H1}).
This implies that any invariant measures supported on $\Delta_{\partial,1}$ are attracting with respect to $\mathcal D$ (cf.~Section~4.1).
In particular, by considering again $V_{\alpha, \beta, \gamma} = \alpha V_0 + \beta V_1 + \gamma V_2$
and
$H_{\alpha, \beta, \gamma} =  \mathcal L V_{\alpha, \beta, \gamma} = \alpha H_0 + \beta H_1 + \gamma H_2$,
we may choose $\alpha$ and $\gamma$ sufficiently small and $\beta$ sufficiently large to obtain that the boundary $\partial \mathcal D$ as a whole is attracting in the sense of Theorem~\ref{thm:prep} (potentially up to exclusion of $(\mathbf p^*, \mathbf q^*)$ if it is a corner point).

It remains to show that mass is only accumulating at $\Delta_{\partial,1}$.
We introduce 
$ X_1 = \sum_{i \in I} x_i, X_2 = \sum_{i \in I^c} x_i,$
such that $X_1 + X_2 =1$
and
$ Y_1 = \sum_{j \in J} y_j, Y_2 = \sum_{j \in J^c} y_j,$
such that $Y_1 + Y_2 =1$.
Hence, we can split the boundary $\partial \mathcal D$ into four branches connecting the four ``corners"
\begin{align*}
C_{1,0}&=(X_1=1, Y_1 =0), \quad C_{0,0} =(X_1=0, Y_1 =0), \\
 C_{1,1}&=(X_1=1, Y_1 =1), \quad C_{0,1}=(X_1 =0, Y_1 =1).
\end{align*}
For the branch connecting $C_{0,0}$ and $C_{0,1}$ we
consider 
$$V_0^y (x,y) := \sum_{j \in J} (- q_j) \ln y_j$$
as a Lyapunov function for $C_{0,0}$ and 
$$V_1^y (x,y) := \sum_{j \in J^c} (- q_j^*) \ln y_j$$
as a Lyapunov function for $C_{0,1}$. 
Then $L V_0^y = H_0^y < 0$ at $C_{0,0}$ by~\eqref{eq:cond_small_noise_gen_y} and $L V_1^y = H_1^y > 0$ at $C_{0,1}$ by the considerations on $H_1$ in the previous paragraph.
Similarly, for $C_{0,0}$ and $C_{1,0}$, we can choose
$$V_0^x (x,y) := \sum_{i\in I} (- p_i) \ln x_i, \quad V_1^x (x,y) := \sum_{i \in I^c} (- p_i^*) \ln x_i$$
and check $L V_0^x = H_0^x < 0$ at $C_{0,0}$ by~\eqref{eq:cond_small_noise_gen_x} and $L V_1^x = H_1^x > 0$ at $C_{1,0}$.
Analogously, we check $L V_0^y = H_0^y < 0$ at $C_{1,0}$  and $L V_1^y = H_1^y > 0$ at $C_{1,1}$, and also $L V_0^x = H_0^x < 0$ at $C_{0,1}$  and $L V_1^x = H_1^x > 0$ at $C_{1,1}$. 
Hence, the only fully attracting part of the boundary is $\Delta_{\partial,1}$ where therefore the stochastic flow from the interior has to accumulate, see Figure~\ref{fig:boundary}.

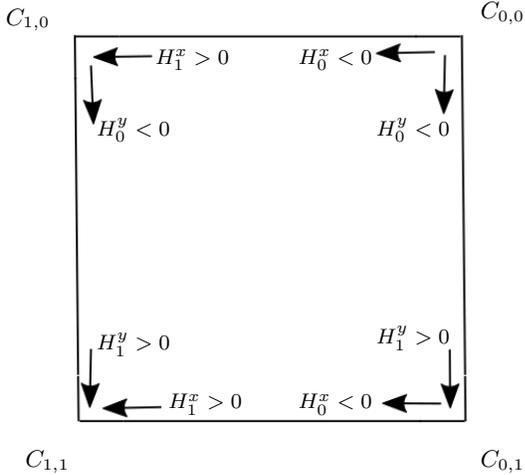
\begin{figure}[ht]
        \centering
  		\begin{overpic}[width=1\linewidth]{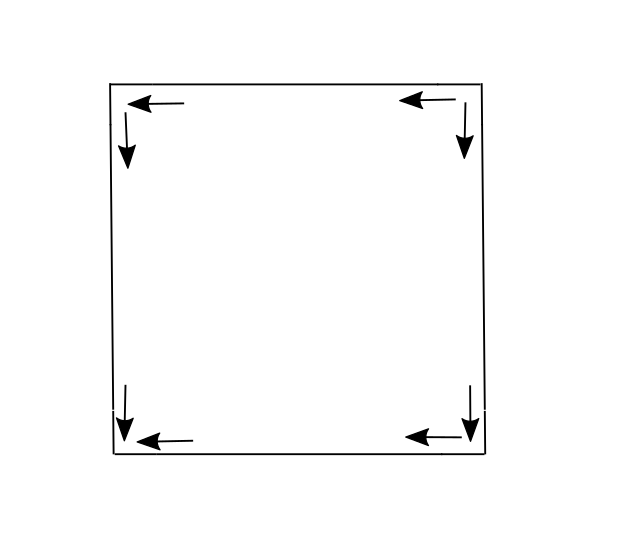}   
\put(7,75){\small $C_{1,0}$}
\put(21,58){\footnotesize $H_0^y < 0$}
\put(64,58){\footnotesize $H_0^y < 0$}
\put(30,69){\footnotesize $H_1^x > 0$}
\put(52,16){\footnotesize $H_0^x < 0$}
\put(64,26){\footnotesize $H_1^y > 0$}
\put(21,25){\footnotesize $H_1^y > 0$}
\put(52,69){\footnotesize $H_0^x < 0$}
\put(32,16){\footnotesize $H_1^x > 0$}	
\put(80,76){\small $C_{0,0}$}
                \put(10,7){\small $C_{1,1}$}
                
                \put(80,7){\small $C_{0,1}$}
        \end{overpic}
        \caption{Structuring the boundary $\partial \mathcal D$ in the small noise case with no interior Nash equilibrium.}
        \label{fig:boundary}
		\end{figure}%

%
From treating the case with non-interior Nash-equilibrium, we can now immediately deduce (a)(i) for the general situation.
Statement (a)(ii) follows from iteratively applying the arguments from above to each invariant subspace of the boundary $\partial \mathcal D$.
Statement (b) already followed from the fact that \eqref{eq:lowerbound} holds uniformly on $\partial \mathcal D$, and statement (c) follows immediately from our treatment of the case with non-interior Nash-equilibrium.
The conditions given in statement (d) have been verified in the course of the proof.
\end{proof}

We discuss some high level remarks and intuition derived by the theorem:


\begin{remark}
\begin{enumerate}[(i)]
\item The ergodic measures supported at the corners, $\delta_{v_{i,j}}$, are, depending on the noise strength, either saddles or stable nodes or partially stable (non-hyperbolic) within $\partial \mathcal D$, organizing the stochastic dynamics. We see this in more detail in the examples of Section~\ref{sec:ill_example}.
%
 \item  Note that it is easy to check that conditions~\eqref{eq:cond_large_noise} and~\eqref{eq:cond_small_noise} are mutually exclusive. In particular, one may assume without loss of generality that $\sum_{i \in I} p_i\{\mathbf{ A q}\}_i =0$, since the structure of the game is not influenced by this value, such that the conditions read more simply. 
For intermediary situations between such a large and small noise regime, it appears difficult to make general predictions due to the possibility of very heterogeneous behavior at different parts of the boundary. An even finer description of such scenarios may be an interesting point for future work.

\item In the case of zero-sum games with interior Nash, deterministic replicator dynamics is Poincar\'{e} recurrent with the trajectories cycling around the Nash equilibrium at a constant Kullback-Leibler divergence.  In contrast, stochastic replicator dynamics diverges to the boundary spending most of its time very close to pure strategy outcomes. 
\item In the case of zero-sum games without interior Nash, deterministic replicator dynamics converges to the sub-simplex defined by the Nash equilibrium of maximal support. In that subspace, the dynamic is once again Poincar\'{e} recurrent. The stochastic replicator dynamics under small enough stochastic noise converges to the correct sub-simplex, but in that subsimplex it once again  concentrates on its corners. 
Showing this has been the crucial intricacy of our analysis.
\item 
The stochastic replicator and deterministic replicator dynamics behave quite differently to each other in zero-sum games. It is natural to interpret a mixed probability distribution as describing uncertainty about the optimal strategy from the perspective of an agent. In contrast, a near-pure probability distribution, i.e.~one that puts almost all its probability mass on a single action shows an agent's confidence that this corresponds to an optimal action for them. Thus, deterministic replicator dynamics will typically lead to a cyclic-like behavior over mixed outcomes, exhibiting uncertainty about which is the correct, optimal action. In contrast, stochastic replicator dynamics spends most of its time around pure, deterministic strategy outcomes, exhibiting a false confidence about which is the correct action to take. 
\end{enumerate}

\end{remark}

Next, we will examine the stochastic version of Example~\ref{ex:deterministic}  demonstrating how our theoretical results are in good agreement with simulations in concrete zero-sum games.

\section{Examples revisited}
\label{sec:ill_example}

\subsection{Matching Pennies} \label{sec:MP_2times2_noise}
Firstly, we consider model~\eqref{eq:MP_2times2_determ} perturbed by Gaussian noise with intensity parameters $\sigma, \eta \geq 0$
\begin{equation}  \label{ex:bos_ito}
\begin{array}{rl}
\rmd X_1(t) &= 2 X_1(t)(1- X_1(t)(2 Y_1(t)-1)
\rmd t \\
&+ \sigma X_1(t) (1-X_1(t)) \,  \rmd W_1(t)\,,\\
\rmd Y_1(t) &= 2 Y_1(t)(1-Y_1(t)(1 - 2 X_1(t)) \,  \rmd t \\
&+   \eta Y_1(t) (1-Y_1(t)) \, \rmd \tilde W_1(t)\,,
\end{array}
\end{equation}
where $X_2 = 1 - X_1$, $Y_2 = 1 -Y_1$.
Note that this system is an example of model~\eqref{eq:stoc2} after reduction to two equations as explained in Appendix~\ref{App:second_derivatives}.
Recall that the system without noise, i.e.~$\sigma=\eta =0$, is a zero-sum game with interior Nash equilibrium at
$\mathbf p =\left( \frac{1}{2}, \frac{1}{2}\right), \mathbf q = \left( \frac{1}{2}, \frac{1}{2}\right)$,
and constant of motion
$V(x,y) = -\frac{1}{2} \left( \ln x + \ln(1-x) + \ln(y) + \ln(1-y)\right).$
Denoting the generator of the Markov semigroup associated with the SDE~\eqref{ex:bos_ito} by $\mathcal L$, we obtain 
$H(x,y) = \mathcal{L} V(x,y) = \frac{\sigma^2}{4} \left((1-x)^2 + x^2 \right) + \frac{\eta^2}{4} \left((1-y)^2 + y^2 \right)$.
Hence, we can conclude that on the boundary $\partial \mathcal D$ of the domain $\mathcal D := (0,1) \times (0,1)$, we have, as a special case of~\eqref{eq:lowerbound}, 
$H = \mathcal L V \geq \min \{ \frac{\sigma^2}{4}, \frac{\eta^2}{4} \} > 0.$ 
As an example for Theorem~\ref{thm:general}, we can derive that, for $\sigma, \eta > 0$, 
 any invariant measure $\mu$ for the process induced by equation~\eqref{ex:bos_ito} on $\overline{\mathcal{D}} = [0,1] \times [0,1]$ is supported on the boundary $ \{0,1\} \times [0,1] \cup [0,1] \times \{0,1\} $.
In particular, any such measure is ergodic since it is given by a convex combination of the ergodic Dirac measures
$\delta_{i,j}$ for $i,j \in \{0,1\}$, supported on the corners.
Additionally, we now consider in more detail the relation between these Dirac measures via the limiting dynamics on the boundary and determine the physical measure, i.e.~the invariant measure which will be approximated by almost all trajectories starting in $\mathcal D$.
For this purpose, we restrict equation~\eqref{ex:bos_ito} to the four parts of the process-invariant boundary $\partial \mathcal D$, and use $- \ln y, - \ln(1-y), - \ln x, - \ln (1-x)$ as respective Lyapunov functions.
For example, we observe that on $\{0\} \times [0,1]$ equation~\eqref{ex:bos_ito} reduces to
\begin{equation*}
\rmd Y_1(t) = 2 \left[Y_1(t)(1-Y_1(t)) \right]\rmd t 
 + \eta Y_1(t)(1-Y_1(t)) \, \rmd \tilde W_1(t)\,.
\end{equation*}
Applying the generator $\mathcal L_y$ of the corresponding process to the respective Lyapunov functions gives
\begin{align*}
H_{0}(y)&:=\mathcal L_y (- \ln y) = -2(1-y) + \frac{\eta^2}{2}(1-y)^2,\\
 H_{1}(y)&:=\mathcal L_y (- \ln(1- y)) = 2y + \frac{\eta^2}{2}(1-y)^2
\end{align*}
and, in particular,
$ H_{0}(0)= -2 + \frac{\eta^2}{2}, \quad H_{1}(1)= 2 + \frac{\eta^2}{2}.$
Observe that, restricting to $\{0\} \times [0,1]$, the Nash equilibrium of maximum support is at $Y_1=1$ and the anti-equilibrium at $Y_1=0$, and we are in the situation of Theorem~\ref{thm:general} (c).
In particular, the measure $\delta_{0,0}$ is repelling for $ \eta < 2$ and attracting for $\eta > 2$, and $\delta_{0,1}$ is always attracting.
Similarly, we observe that on $ [0,1] \times \{0\} $ equation~\eqref{ex:bos_ito} reduces to
\begin{align*}
\rmd X_1(t) = &-2 \left[X_1(t)(1-X_1(t)) \right]\rmd t 
 \\ &+ \sigma X_1(t)(1-X_1(t)) \, \rmd  W_1(t)\,,
\end{align*}
such that the generator $\mathcal L_x$ fulfills
\begin{align*}
 \hat H_{0}(x)&:=\mathcal L_y (- \ln (1-x)) = -2x + + \frac{\sigma^2}{2}(1-x)^2,\\
 \hat H_{1}(x)&:=\mathcal L_x (- \ln x) = 2(1-x) + \frac{\sigma^2}{2}(1-x)^2,
\end{align*}
and, in particular,
$ \hat H_{0}(1)= -2 + \frac{\sigma^2}{2}, \quad \hat H_{1}(0)= 2 + \frac{\eta^2}{2},$
such that, restricting to $ [0,1] \times \{0\} $, the measure $\delta_{0,0}$ is always attracting and the measure $\delta_{1,0}$ is repelling for $ \sigma < 2$ and attracting for $\sigma > 2$.
Checking the other cases analogously, we obtain, for $\eta, \sigma < 2$, that the Dirac measures
$\delta_{i,j}$ are saddles forming the heteroclinic cycle
$ \delta_{0,0} \rightarrow \delta_{0,1} \rightarrow \delta_{1,1} \rightarrow \delta_{1,0} \to \delta_{0,0},$
mirroring the periodic orbits of the deterministic dynamics. The physical invariant measure (see section~4.1)
is then a linear combination of all Dirac measures, as illustrated in Figure~\ref{fig:2times2_stoch}.
\begin{figure}[htbp]
\centering
\begin{subfigure}[b]{0.45\textwidth}
\includegraphics[width=0.9\textwidth]{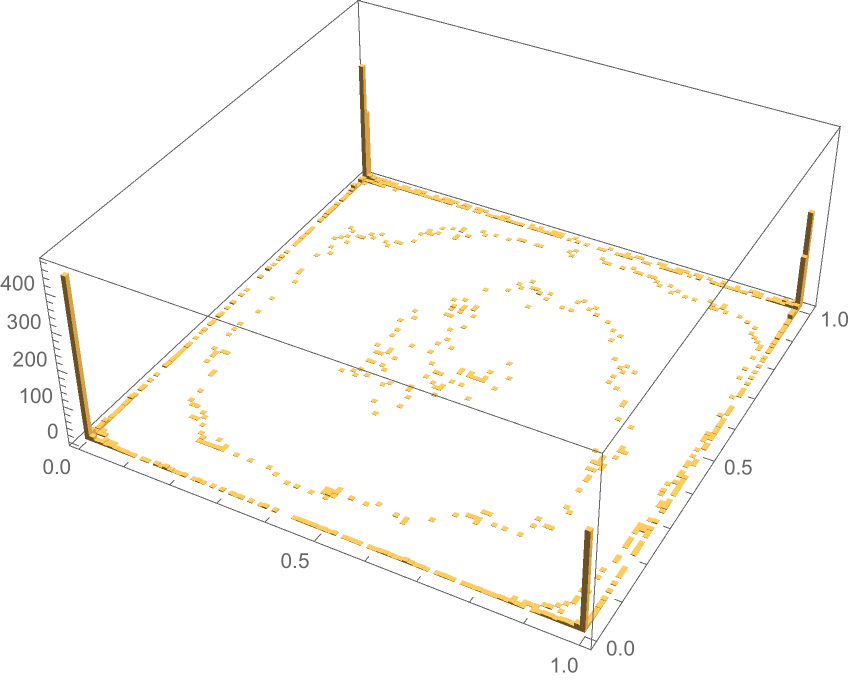}
  \caption{Stochastic replicator dynamics in MP ($\eta, \sigma = 0.2$)}
  \label{fig:2times2_stoch}
\end{subfigure}
\begin{subfigure}[b]{0.45\textwidth}
\hspace{10pt}
\includegraphics[width=0.8\textwidth]{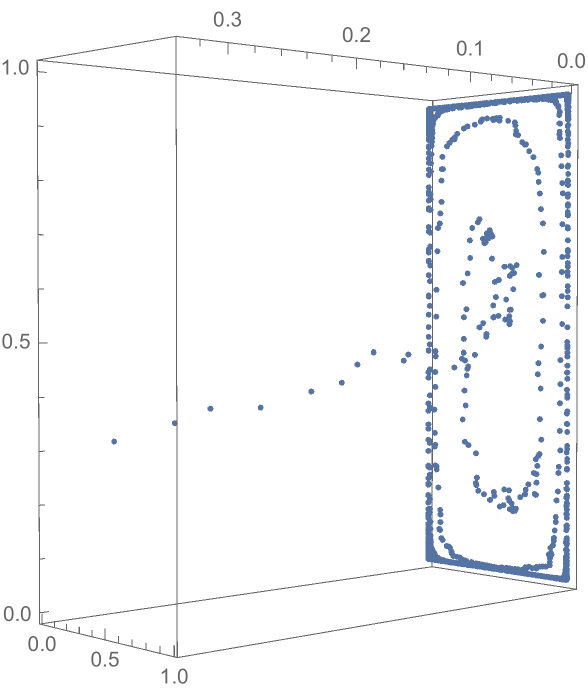}
  \caption{Stochastic  dynamics with $\sigma_1, \sigma_2, \sigma_3, \tilde \eta = 0.2$}
  \label{fig:3times2_stoch}
\end{subfigure}
\caption{(a) Stochastic replicator dynamics~\eqref{ex:bos_ito} in Matching Pennies (MP). We depict a histogram of the support of the ergodic, in fact physical, invariant measure whose support is approximated by the simulated trajectory. Unlike deterministic replicator dynamics, which cycles around the Nash equilibrium (0.5, 0.5) at a constant Kullback-Leibler distance, stochastic replicator dynamics concentrates on the corners of the simplex.
(b) Stochastic replicator dynamics~\eqref{ex:bos_ito_2}. These are derived from a $3\times2$ zero-sum game without interior Nash. The three axes correspond to the probabilities player $X$ assigns to her first and third strategies and the probability that player $Y$ assigns to her first strategy. 
 We depict a histogram of the support of the ergodic invariant measure whose support is approximated by the simulated trajectory. Similarly to the deterministic replicator dynamics, the stochastic replicator dynamics gets attracted to the boundary of the simplex that is defined by the equilibrium of maximum support. In that sub-simplex and unlike the deterministic replicator dynamic that would cycle around the equilibrium, the 
stochastic replicator dynamics  concentrates on the corners of that  sub-simplex.}
\end{figure}

\subsection{Example of game without interior Nash equilibrium} \label{sec:ex3times2}
As in Example~\ref{ex:deterministic}, we now add one dimension and consider the stochastic version of model~\eqref{eq:MP_3times2_determ}, again simplifying the $Y$-components to $Y_1$ and $1-Y_1$ but now writing out the equations for all $X$-components, i.e., also for $X_3$.
Hence, in terms of ~\eqref{eq:stoc2} we summarize the noise for $Y_1$ into $\tilde \eta \, \rmd \tilde W(t)$ and obtain
\begin{equation}  \label{ex:bos_ito_2}
\begin{array}{rl}
\rmd X_1(t) &= [ X_1(t)(1- X_1(t)(2 Y_1(t)-1) \\
&+ X_1(t) X_2(t)(2 Y_1(t)-1) + 2 X_1(t)X_3(t) ]
\rmd t \\
 &+ \sigma_1 X_1(t) (X_1(t)-1) \,  \rmd W_1(t)   \\&+\sigma_2 X_1(t) X_2(t) \,\rmd W_2(t) + \sigma_3 X_1(t) X_3(t) \,  \rmd W_3(t)\,,\\
\rmd X_2(t) &= [ X_2(t)(1- X_2(t)(1 -2 Y_1(t)) \\
&+ X_1(t) X_2(t)(1-2 Y_1(t)) + 2 X_2(t)X_3(t) ] 
\rmd t \\
 &+  \sigma_2 X_2(t) (X_2(t)-1) \, \rmd  W_2(t) + \\& \sigma_1 X_2(t) X_1(t) \, \rmd  W_1(t) + \sigma_3 X_2(t) X_3(t) \,  \rmd W_3(t)\,,\\
 \rmd X_3(t) &= [ -2 X_3(t)(1- X_3(t)) - X_3(t) X_2(t)(1-2 Y_1(t)) \\&
 - X_3(t)X_1(t)(2Y_1(t)-1) ]
\rmd t \\
 &+  \sigma_3 X_3(t) (X_3(t)-1) \, \rmd  W_3(t) \\
 &+ \sigma_1 X_3(t) X_1(t) \, \rmd  W_1(t) + \sigma_2 X_2(t) X_3(t) \,  \rmd W_2(t)\,,\\
\rmd Y_1(t) &= \left[Y_1(t)(1-Y_1(t)(2X_2(t) - 2 X_1(t)\right] \,  \rmd t \\
&+  \tilde \eta Y_1(t) (1-Y_1(t)) \, \rmd \tilde W(t)\,.
\end{array}
\end{equation}
where $X_1(t) + X_2(t) + X_3(t)=1$ is satisfied at any time step. 
Recall that the Nash equilibrium of maximal support is 
$\mathbf p =\left( \frac{1}{2}, \frac{1}{2}, 0\right), \mathbf q = \left( \frac{1}{2}, \frac{1}{2}\right)$.
In this case, the Lyapunov function $V_0$~\eqref{V0} is given by
$$V_0(\mathbf{X, Y}) :=  - \frac{1}{2}\ln X_1 -  \frac{1}{2}\ln X_2 - \frac{1}{2} \ln Y_1 -\frac{1}{2} \ln (1-Y_1).$$
Recalling from~\eqref{eq:JK} that
$H_0(\mathbf{X,Y}) = L_0(\mathbf{X,Y}) + K_0(\mathbf{X,Y})$,
we obtain
\begin{align*}
L_0(\mathbf{X,Y})&= - \frac{1}{2} (1-X_1) (2 Y_1 -1) - \frac{1}{2}X_2 (2Y_1-1) \\
&- \frac{1}{2}(1-X_2) (1- 2Y_1) - \frac{1}{2}X_1 (1- 2Y_1)  -2 X_3 \\
&= -2 X_3 + X_1(2Y_1-1) + X_2 (1- 2Y_1) \,,\\
K_0(\mathbf{X,Y})&= \frac{1}{2}\bigg( \frac{\sigma_1^2}{2}(1-X_1)^2 + \frac{\sigma_2^2}{2} X_2^2 + \frac{\sigma_3^2}{2}X_3^2 \\
&+ \frac{\sigma_2^2}{2}(1-X_2)^2 + \frac{\sigma_1^2}{2} X_1^2 + \frac{\sigma_3^2}{2}X_3^2\bigg)\\
&+ \frac{\tilde \eta^2}{2} Y_1^2 + \frac{\tilde \eta^2}{2} (1-Y_1)^2  \,.
\end{align*}
Hence, at $\mathbf{X}=(0,0,1)$, we obtain
$$ H_0(\mathbf{X,Y}) = - 2 + \sum_{j=1}^3 \frac{\sigma_j^2}{2} + \frac{\tilde \eta^2}{2} Y_1^2 + \frac{\tilde \eta^2}{2} (1-Y_1)^2,$$
whose sign depends on the size of $\sigma_j, \tilde \eta$, as described in Theorem~\ref{thm:general}, determining whether
$$\Delta_{\partial,2} := \{ (\mathbf{X,Y}) \in  \overline{\mathcal D}\,:\,  X_1=X_2=0, X_3 =1 \}$$
is attracting or repelling. Note that for $\sigma_j, \tilde \eta < 1$, the noise is sufficiently small such that $H_0 < 0$, and we are in the situation of Theorem~\ref{thm:general} (c)(ii).

In more detail, in this case $V_1$ \eqref{V_noint} is given by
$V_1(\mathbf{X, Y}) :=  - \ln X_3$
such that we have for $H_1$ as in~\eqref{eq:H1}
\begin{equation*}
H_1(\mathbf{X, Y}) = \mathcal L  V_1(\mathbf{X, Y}) =  L_1(\mathbf{X, Y}) + K_1(\mathbf{X, Y}),
 \end{equation*}
 where
\begin{equation*}
L_1(\mathbf{X, Y}) = 2(1-X_3) + X_2(1-2Y_1) + X_1(2Y_1-1),
\end{equation*}
and
\begin{equation*}
K_1(\mathbf{X, Y}) = \frac{\sigma_3^2}{2}(1-X_3)^2 + \frac{\sigma_2^2}{2}X_2^2 + \frac{\sigma_1^2}{2}X_1^2.
\end{equation*}
Hence, for $X_3=0$ we have
$$L_1(\mathbf{X, Y}) = 2 +(2Y_1-1)(X_1-X_2) \geq 1,$$
and therefore uniformly
$$H_1 \geq 1+ \frac{\sigma_3^2}{2} + \frac{1}{4} \min\{\sigma_1^2, \sigma_2^2\} > 0.$$
Hence, as predicted in Theorem~\eqref{thm:general}, the subspace
$\Delta_{\partial,1} := \{ (\mathbf{X,Y}) \in  \overline{\mathcal D}\,:\, X_3 =0 \}$ 
is always attracting, i.e.~all attracting invariant measures are supported on $\Delta_{\partial,1}$ which corresponds to the support of the maximal Nash equilibrium $(\mathbf p, \mathbf q)$. On this subspace the stochastic dynamics have the same characterization as elaborated in Section~5.1.
The behavior is illustrated in Figure~\ref{fig:3times2_stoch}.

\section{Conclusion}

We study a general stochastic variant of replicator dynamics and use techniques from stochastic differential equations and stochastic stability theory to characterize its invariant measures in general zero-sum games.
Organizing the dynamical system at hand by its invariant measures is the classical approach of ergodic theory to deterministic as well as stochastic dynamics and, by that, offers a highly suitable framework for comparing both scenarios which we suggest to employ also in future efforts.
Whereas the deterministic bimatrix replicator dynamics are characterized by infinitely many invariant measures supported on periodic orbits in the interior of the simplex, this structure is broken down by the stochastic perturbations. The fact that this happens as such is not surprising but the concrete outcome is intriguing:
the  ergodic invariant measures found in our model are supported on pure strategy profiles even if the Nash equilibrium is fully mixed.  Thus, the emergent behavior is in contrast both to the Nash equilibrium prediction as well as the behavior suggested by the standard deterministic replicator equation, i.e.~recurrence/cycles, in the sense that the uncertainty drives players away from mixed strategies. 
In particular, depending on the noise strength and the position of the potentially non-interior Nash equilibrium, we can determine the \emph{physical} ergodic invariant measure, whose pure strategy support is approached by the time averages from almost any starting point in the state space.


\section*{Acknowledgements}
M.E. thanks the DFG SPP 2298 "Theoretical Foundations of Deep Learning" for supporting his research. He has been additionally supported by Germany’s Excellence Strategy – The Berlin Mathematics Research Center MATH+ with the number EXC-2046/1, project ID:
390685689 (subprojects AA1-8 and AA1-18). Furthermore, M.E. thanks the DFG CRC 1114 and the Einstein Foundation (IPF-2021-651) for support.

\bibliographystyle{abbrv}
\bibliography{mybibfile}

\begin{thebibliography}{10}

\bibitem{abe2022mutation}
K.~Abe, M.~Sakamoto, and A.~Iwasaki.
\newblock Mutation-driven follow the regularized leader for last-iterate
  convergence in zero-sum games.
\newblock In {\em Uncertainty in Artificial Intelligence}, pages 1--10. PMLR,
  2022.

\bibitem{adolphs2018local}
L.~Adolphs, H.~Daneshmand, A.~Lucchi, and T.~Hofmann.
\newblock Local saddle point optimization: {A} curvature exploitation approach.
\newblock In {\em The 22nd International Conference on Artificial Intelligence
  and Statistics, {AISTATS} 2019, 16-18 April 2019, Naha, Okinawa, Japan},
  pages 486--495, 2019.

\bibitem{Arora05themultiplicative}
S.~Arora, E.~Hazan, and S.~Kale.
\newblock The multiplicative weights update method: a meta-algorithm and
  applications.
\newblock {\em Theory of Computing}, 8(1):121--164, 2012.

\bibitem{eatwell1987new}
R.~J. Aumann.
\newblock Game theory.
\newblock In J.~Eatwell, M.~Milgate, and P.~Newman, editors, {\em The new
  Palgrave: a dictionary of economics}. London (UK) Macmillan, 1987.

\bibitem{bailey2020finite}
J.~P. Bailey, G.~Gidel, and G.~Piliouras.
\newblock Finite regret and cycles with fixed step-size via alternating
  gradient descent-ascent.
\newblock {\em CoRR}, abs/1907.04392, 2019.

\bibitem{BaileyEC18}
J.~P. Bailey and G.~Piliouras.
\newblock Multiplicative weights update in zero-sum games.
\newblock In {\em ACM Conference on Economics and Computation}, 2018.

\bibitem{bailey2019multi}
J.~P. {Bailey} and G.~{Piliouras}.
\newblock {Multi-Agent Learning in Network Zero-Sum Games is a {H}amiltonian
  System}.
\newblock In {\em AAMAS}, 2019.

\bibitem{bauer2019stabilization}
J.~Bauer, M.~Broom, and E.~Alonso.
\newblock The stabilization of equilibria in evolutionary game dynamics through
  mutation: mutation limits in evolutionary games.
\newblock {\em Proceedings of the Royal Society A}, 475(2231):20190355, 2019.

\bibitem{BenaimHofbauerSandholm}
M.~Bena\"{\i}m, J.~Hofbauer, and W.~H. Sandholm.
\newblock Robust permanence and impermanence for stochastic replicator
  dynamics.
\newblock {\em J. Biol. Dyn.}, 2(2):180--195, 2008.

\bibitem{benaim2006stochastic}
M.~Bena{\"\i}m, J.~Hofbauer, and S.~Sorin.
\newblock Stochastic approximations and differential inclusions, part ii:
  Applications.
\newblock {\em Mathematics of Operations Research}, 31(4):673--695, 2006.

\bibitem{BenaimStrickler}
M.~Bena\"{\i}m and E.~Strickler.
\newblock Random switching between vector fields having a common zero.
\newblock {\em Ann. Appl. Probab.}, 29(1):326--375, 2019.

\bibitem{Brown1951}
G.~W. Brown.
\newblock Iterative solutions of games by fictitious play.
\newblock In T.~C. Coopmans, editor, {\em Activity Analysis of Productions and
  Allocation}, 374-376. Wiley, 1951.

\bibitem{Cesa06}
N.~Cesa-Bianchi and G.~Lugosi.
\newblock {\em Prediction, Learning, and Games}.
\newblock Cambridge University Press, 2006.

\bibitem{cheung2018multiplicative}
Y.~K. Cheung.
\newblock Multiplicative weights updates with constant step-size in graphical
  constant-sum games.
\newblock In {\em Advances in Neural Information Processing Systems}, pages
  3528--3538, 2018.

\bibitem{cheung2019vortices}
Y.~K. Cheung and G.~Piliouras.
\newblock Vortices instead of equilibria in minmax optimization: Chaos and
  butterfly effects of online learning in zero-sum games.
\newblock In {\em COLT}, 2019.

\bibitem{cheung2020chaos}
Y.~K. Cheung and G.~Piliouras.
\newblock Chaos, extremism and optimism: Volume analysis of learning in games.
\newblock In {\em NeurIPS}, 2020.

\bibitem{dantzig1951proof}
G.~B. Dantzig.
\newblock A proof of the equivalence of the programming problem and the game
  problem.
\newblock In {\em Activity {A}nalysis of {P}roduction and {A}llocation}, Cowles
  Commission Monographs, No. 13, pages 330--335. John Wiley \& Sons, Inc., New
  York, N.Y.; Chapman \& Hall, Ltd., London, 1951.

\bibitem{daskalakis2018training}
C.~Daskalakis, A.~Ilyas, V.~Syrgkanis, and H.~Zeng.
\newblock Training {GAN}s with optimism.
\newblock In {\em ICLR}, 2018.

\bibitem{daskalakis2018limit}
C.~Daskalakis and I.~Panageas.
\newblock The limit points of (optimistic) gradient descent in min-max
  optimization.
\newblock In {\em Advances in Neural Information Processing Systems}, pages
  9236--9246, 2018.

\bibitem{foster}
D.~Foster and P.~Young.
\newblock Stochastic evolutionary game dynamics.
\newblock {\em Theoret. Population Biol.}, 38(2):219--232, 1990.

\bibitem{freund1999adaptive}
Y.~Freund and R.~E. Schapire.
\newblock Adaptive game playing using multiplicative weights.
\newblock {\em Games and Economic Behavior}, 29(1-2):79--103, 1999.

\bibitem{fudenberg1992evolutionary}
D.~Fudenberg and C.~Harris.
\newblock Evolutionary dynamics with aggregate shocks.
\newblock {\em Journal of Economic Theory}, 57(2):420--441, 1992.

\bibitem{gidel2019a}
G.~Gidel, H.~Berard, G.~Vignoud, P.~Vincent, and S.~Lacoste-Julien.
\newblock A variational inequality perspective on generative adversarial
  networks.
\newblock In {\em ICLR}, 2019.

\bibitem{goodfellow2014generative}
I.~Goodfellow, J.~Pouget-Abadie, M.~Mirza, B.~Xu, D.~Warde-Farley, S.~Ozair,
  A.~Courville, and Y.~Bengio.
\newblock Generative adversarial nets.
\newblock In {\em Advances in neural information processing systems}, pages
  2672--2680, 2014.

\bibitem{HeningNgyuen2018}
A.~Hening and D.~H. Nguyen.
\newblock Coexistence and extinction for stochastic {K}olmogorov systems.
\newblock {\em Ann. Appl. Probab.}, 28(3):1893--1942, 2018.

\bibitem{HeningNguyenChesson}
A.~Hening, D.~H. Nguyen, and P.~Chesson.
\newblock A general theory of coexistence and extinction for stochastic
  ecological communities.
\newblock {\em J. Math. Biol.}, 82(6):Paper No. 56, 76, 2021.

\bibitem{HeningNguyenSchreiber}
A.~Hening, D.~H. Nguyen, and S.~J. Schreiber.
\newblock A classification of the dynamics of three-dimensional stochastic
  ecological systems.
\newblock {\em Ann. Appl. Probab.}, 32(2):893--931, 2022.

\bibitem{Hofbauer96}
J.~Hofbauer.
\newblock Evolutionary dynamics for bimatrix games: a {H}amiltonian system?
\newblock {\em J. Math. Biol.}, 34(5-6):675--688, 1996.

\bibitem{HofbauerImhof}
J.~Hofbauer and L.~A. Imhof.
\newblock Time averages, recurrence and transience in the stochastic replicator
  dynamics.
\newblock {\em Ann. Appl. Probab.}, 19(4):1347--1368, 2009.

\bibitem{hsieh2020limits}
Y.-P. Hsieh, P.~Mertikopoulos, and V.~Cevher.
\newblock The limits of min-max optimization algorithms: Convergence to
  spurious non-critical sets.
\newblock In M.~Meila and T.~Zhang, editors, {\em Proceedings of the 38th
  International Conference on Machine Learning}, volume 139 of {\em Proceedings
  of Machine Learning Research}, pages 4337--4348. PMLR, 18--24 Jul 2021.

\bibitem{Imh05}
L.~A. Imhof.
\newblock The long-run behavior of the stochastic replicator dynamics.
\newblock {\em Ann. Appl. Probab.}, 15(1B):1019--1045, 2005.

\bibitem{jin2017escape}
C.~Jin, R.~Ge, P.~Netrapalli, S.~M. Kakade, and M.~I. Jordan.
\newblock How to escape saddle points efficiently.
\newblock In {\em Proceedings of the 34th International Conference on Machine
  Learning-Volume 70}, pages 1724--1732. JMLR. org, 2017.

\bibitem{Khasminskii80}
R.~Khasminskii.
\newblock {\em Stochastic stability of differential equations}, volume~66 of
  {\em Stochastic Modelling and Applied Probability}.
\newblock Springer, Heidelberg, second edition, 2012.
\newblock With contributions by G. N. Milstein and M. B. Nevelson.

\bibitem{Kleinberg09multiplicativeupdates}
R.~Kleinberg, G.~Piliouras, and {\'E}.~Tardos.
\newblock Multiplicative updates outperform generic no-regret learning in
  congestion games.
\newblock In {\em ACM Symposium on Theory of Computing (STOC)}, 2009.

\bibitem{leonardos2021exploration}
S.~Leonardos, G.~Piliouras, and K.~Spendlove.
\newblock Exploration-exploitation in multi-agent competition: convergence with
  bounded rationality.
\newblock {\em Advances in Neural Information Processing Systems},
  34:26318--26331, 2021.

\bibitem{Akin84}
V.~Losert and E.~Akin.
\newblock Dynamics of games and genes: Discrete versus continuous time.
\newblock {\em Journal of Mathematical Biology}, 1983.

\bibitem{mertikopoulos2019optimistic}
P.~Mertikopoulos, B.~Lecouat, H.~Zenati, C.-S. Foo, V.~Chandrasekhar, and
  G.~Piliouras.
\newblock Optimistic mirror descent in saddle-point problems: Going the
  extra(-gradient) mile.
\newblock In {\em ICLR}, 2019.

\bibitem{mertikopoulos2010emergence}
P.~Mertikopoulos and A.~L. Moustakas.
\newblock The emergence of rational behavior in the presence of stochastic
  perturbations.
\newblock {\em The Annals of Applied Probability}, 20(4):1359--1388, 2010.

\bibitem{GeorgiosSODA18}
P.~Mertikopoulos, C.~Papadimitriou, and G.~Piliouras.
\newblock Cycles in adversarial regularized learning.
\newblock In {\em ACM-SIAM Symposium on Discrete Algorithms}, 2018.

\bibitem{mescheder2018training}
L.~M. Mescheder, A.~Geiger, and S.~Nowozin.
\newblock Which training methods for gans do actually converge?
\newblock In {\em International Conference on Machine Learning}, 2018.

\bibitem{Nisan:2007:AGT:1296179}
N.~Nisan, T.~Roughgarden, E.~Tardos, and V.~V. Vazirani.
\newblock {\em Algorithmic Game Theory}.
\newblock Cambridge University Press, New York, NY, USA, 2007.

\bibitem{papadimitriou2019game}
C.~Papadimitriou and G.~Piliouras.
\newblock Game dynamics as the meaning of a game.
\newblock {\em ACM SIGecom Exchanges}, 16(2):53--63, 2019.

\bibitem{perolat2020poincar}
J.~P{\'e}rolat, R.~Munos, J.-B. Lespiau, S.~Omidshafiei, M.~Rowland, P.~A.
  Ortega, N.~Burch, T.~W. Anthony, D.~Balduzzi, B.~D. Vylder, G.~Piliouras,
  M.~Lanctot, and K.~Tuyls.
\newblock From poincar{\'e} recurrence to convergence in imperfect information
  games: Finding equilibrium via regularization.
\newblock {\em ArXiv}, abs/2002.08456, 2020.

\bibitem{piliouras2014optimization}
G.~Piliouras and J.~S. Shamma.
\newblock Optimization despite chaos: Convex relaxations to complex limit sets
  via poincar{\'e} recurrence.
\newblock In {\em Proceedings of the twenty-fifth annual ACM-SIAM symposium on
  Discrete algorithms}, pages 861--873. SIAM, 2014.

\bibitem{Robinson1951}
J.~Robinson.
\newblock An iterative method of solving a game.
\newblock {\em Annals of Mathematics}, 54:296--301, 1951.

\bibitem{Sandholm10}
W.~H. Sandholm.
\newblock {\em Population Games and Evolutionary Dynamics}.
\newblock MIT Press, 2010.

\bibitem{BenaimSchreiberAtchade}
S.~J. Schreiber, M.~Bena\"{\i}m, and K.~A.~S. Atchad\'{e}.
\newblock Persistence in fluctuating environments.
\newblock {\em J. Math. Biol.}, 62(5):655--683, 2011.

\bibitem{Shalev2012}
S.~Shalev-Shwartz.
\newblock Online learning and online convex optimization.
\newblock {\em Foundations and Trends® in Machine Learning}, 4(2):107--194,
  2012.

\bibitem{sorin2009exponential}
S.~Sorin.
\newblock Exponential weight algorithm in continuous time.
\newblock {\em Mathematical Programming}, 116(1):513--528, 2009.

\bibitem{van1991stability}
E.~Van~Damme.
\newblock {\em Stability and perfection of {N}ash equilibria}, volume 339.
\newblock Springer, 1991.

\bibitem{vlatakis2019efficiently}
E.-V. Vlatakis-Gkaragkounis, L.~Flokas, and G.~Piliouras.
\newblock Efficiently avoiding saddle points with zero order methods: No
  gradients required.
\newblock In {\em Advances in Neural Information Processing Systems}, pages
  10066--10077, 2019.

\bibitem{vlatakis2019poincare}
E.-V. Vlatakis-Gkaragkounis, L.~Flokas, and G.~Piliouras.
\newblock Poincar{\'e} recurrence, cycles and spurious equilibria in
  gradient-descent-ascent for non-convex non-concave zero-sum games.
\newblock In {\em Advances in Neural Information Processing Systems 32: Annual
  Conference on Neural Information Processing Systems 2019}, 2019.

\bibitem{Neumann1928}
J.~von Neumann.
\newblock Zur {T}heorie der {G}esellschaftsspiele.
\newblock {\em Mathematische Annalen}, 100:295--300, 1928.

\bibitem{neumann}
J.~von Neumann.
\newblock Theory of games and economic behavior.
\newblock {\em Princeton University Press}, 1944.

\bibitem{Weibull}
J.~W. Weibull.
\newblock {\em Evolutionary Game Theory}.
\newblock MIT Press; Cambridge, MA: Cambridge University Press., 1995.

\bibitem{yazici2018unusual}
Y.~{Yaz{\i}c{\i}}, C.-S. {Foo}, S.~{Winkler}, K.-H. {Yap}, G.~{Piliouras}, and
  V.~{Chandrasekhar}.
\newblock {The Unusual Effectiveness of Averaging in {GAN} Training}.
\newblock {\em ArXiv e-prints}, June 2018.

\end{thebibliography}


\newpage

\appendix
\section{Proofs for deterministic replicator dynamics in zero sum games (Section~\ref{deterministiccase})}
\label{app:determinstic_system}

\subsection{Games with interior equilibria}

\begin{thm}
Consider the flow of the replicator dynamics when applied to a zero-sum game 
 that has an interior (i.e. fully mixed) Nash equilibrium $(\mathbf{p},\mathbf{q})$.
 Then given any (interior) starting point $x_0\in \mathcal D$,
  the cross entropy
\begin{equation*}
V\big((\mathbf{p},\mathbf{q});(\mathbf{x}(t),\mathbf{y}(t)) \big) = - \sum_i p_i\ln x_i(t) - \sum_j q_i\ln y_j(t)
\end{equation*}
   between the Nash equilibrium  $(\mathbf{p},\mathbf{q})$ and the trajectory $(\mathbf{x}(t),\mathbf{y}(t))$ of the system is a \textit{constant of motion}, \textit{i.e.}
    \[\frac{\rmd  V\big((\mathbf{p},\mathbf{q});(\mathbf{x}(t),\mathbf{y}(t)) \big)}{\rmd  t} |_{t=t'}= 0\].
 Otherwise, let $(\mathbf{p},\mathbf{q})$ (resp. $(\mathbf{p^*},\mathbf{q^*})$) be a not fully mixed Nash  equilibrium (resp. anti-equilibrium) of maximal support on the boundary $\partial \mathcal D$; then for each starting point $x_0\in \mathcal D$  
 and for all $t'\geq 0$ we have
 \[\frac{\rmd  V\big((\mathbf{p},\mathbf{q});(\mathbf{x}(t),\mathbf{y}(t)) \big)}{\rmd  t} |_{t=t'}< 0\]
 and 
  \[\frac{\rmd  V\big((\mathbf{p^*},\mathbf{q^*});(\mathbf{x}(t),\mathbf{y}(t)) \big)}{\rmd  t} |_{t=t'}>0.\]  
\end{thm}

\begin{proof}
The support of the state of the system, e.g., the strategies played with positive probability, is clearly an invariant of the flow; hence, it suffices to prove this statement for any arbitrary starting point $\big(\mathbf{x}(0),\mathbf{y}(0)\big) \in \mathcal D$. 
 We examine the derivative of $V\big((\mathbf{p},\mathbf{q});(\mathbf{x}(t),\mathbf{y}(t)) \big) =  - \sum_i p_i\ln x_i(t) - \sum_j q_i\ln y_j(t)$: 

\vspace{-2mm}

\begin{eqnarray*}
\lefteqn{\sum_{i} p_{i}  \frac{d\ln(x_{i})}{dt}+ \sum_{j} q_{j}  \frac{d\ln(y_{j})}{dt}= \sum_{i}p_{i}  \frac{\dot{x}_{i}}{x_{i}}+ \sum_{j}q_{j}  \frac{\dot{y}_{j}}{y_{j}} }\\ 
&=& \sum_{i} p_i \left(\{\mathbf{Ay}\}_i - \mathbf{x^{\top} A y}\right) + \sum_{j} q_j \left(\{\mathbf{Bx}\}_j - \mathbf{y^{\top} B x}\right) \\
& = &\sum_{i} p_i \{\mathbf{Ay}\}_i  - \mathbf{x^{\top} A y} \big(  \sum_{i} p_i \big) + \\
& & \sum_{j} q_j \{\mathbf{Bx}\}_j - \mathbf{y^{\top} B x} \big( \sum_{i} q_i \big) \\
& = &\sum_{i} p_i \{\mathbf{Ay}\}_i +  \sum_{j} q_j \{\mathbf{Bx}\}_j - \big( \mathbf{x^{\top} A y}  + \mathbf{y^{\top} B x}\big)  \\
& = &\sum_{i} p_i \{\mathbf{Ay}\}_i +  \sum_{j} q_j \{\mathbf{Bx}\}_j  ~~\text{(due to property~\eqref{eq:zero_sum})} \\   
& =& - \big( \sum_{j} y_j \{\mathbf{Bp}\}_j  + \sum_{i} x_i \{\mathbf{Aq}\}_i  \big) ~~\text{(due to~\eqref{eq:zero_sum})} \\   
&  \geq & - \big(\sum_{j}  q_j \{\mathbf{Bp}\}_j  + \sum_{i} p_i \{\mathbf{Aq}\}_i  \big)  ~~\text{($(\mathbf p,\mathbf q)$ being a NE)}     \\ 
&= & 0~~\text{(due to~\eqref{eq:zero_sum})}.
\end{eqnarray*}
Considering in more detail the inequality in the penultimate line above, we observe that for each strategy $i$ and each strategy $j$
$$\sum_{k} \big(p_i -x_i\big)A_{i,k} q_k \geq 0, \quad \sum_{k} \big(q_i -y_j\big)B_{j,k} p_k \geq 0,$$
due to the Nash equilibrium property. 
Since the states $x$ and $y$ are fully mixed, we have 
$$\sum_{k} \big(p_i -x_i\big)A_{i,k} q_k + \sum_{k} \big(q_i -y_j\big)B_{j,k} p_k = 0,$$
 if and only if $(\mathbf p, \mathbf q)$ is fully mixed; in particular, this can be seen from the fact that, if the Nash equilibrium is interior, then all unilateral deviations (e.g.~the first agent deviating from $\mathbf p$ to $\mathbf x$) do not affect the utility of the agent who randomizes over strategies of equal expected payoff. 

Hence, if the zero-sum game has no interior Nash equilibrium but only an equilibrium of maximal support, deviations to strategies not in the equilibrium support result in strictly less payoff~\cite{GeorgiosSODA18} and, hence, the inequality above is strict. The analysis in the case of an anti-equilibrium of maximal support is identical under reverting the direction of the inequalities.
\end{proof}

\subsection{Games without interior equilibria}

We commence the analysis with the following  technical lemma, whose proof
can be found in \cite{Akin84} but is also provided here for completeness:

\begin{lem}
\label{eq:derivative}
If $g(t)$ is a twice differentiable function with uniformly bounded second derivative and $\lim_{t\rightarrow\infty}g(t)$ exists and is finite then we have that $\lim_{t\rightarrow\infty}\dot{g}(t)=0$.
\end{lem}

\begin{proof}
Let's denote by $M\geq 1$  an upper bound on the second derivative of $g$.
Suppose that the statement was not true. In this case, we would be able to find a sequence $\{t_n\}$ going to infinity such that $\dot{g}(t_n)$ remains bounded away from zero. In particular, we can assume that $t_{n+1}>t_n+1$ and $\dot{g}(t_n)\geq \epsilon$
for some $0 < \epsilon < 1$ and all $n\in \mathbb{N}$. If we define $g_{2n}=g(t_n)$ and $g_{2n+1}=g(t_n+\frac{\epsilon}{2M})$, then  a first application of the mean value theorem implies that $\dot{g}(t)\geq \frac{\epsilon}{2}$ for $t_n\leq t \leq t_n + \epsilon/2M$. A second application implies that $g_{2n+1}-g_{2n}\geq \frac{\epsilon^2}{4M}>0.$ Hence,  $\lim_{n\rightarrow\infty}g_n$ if it exists, is infinity. Therefore, the same holds for  $\lim_{t\rightarrow \infty}g(t).$
\end{proof}

We are now ready to prove the following asymptotic property for the
orbits of the flow $\Phi$, associated with solution of the replicator equation~\eqref{eq:replicator2}.

\begin{thm}
\cite{piliouras2014optimization}
If $\Phi$ does not have an interior equilibrium, then
given any interior starting point $x\in \mathcal D$,
the orbit $\Phi(x,\cdot)$ converges to the boundary.  
Furthermore, if $(\textbf p, \textbf q)$ is an equilibrium of maximum support (defined as $\Delta_{\partial,1} \subset \partial \mathcal D$),
then the its limit set $\omega(x)$ satisfies  $\omega(x)\subset \text{int}(\Delta_{\partial,1})$. 
\end{thm}

\begin{proof}
From the second part of Theorem~\ref{thm:constant_of_motion}, we have that starting from any fully mixed strategy profile $(\mathbf{x}(0),\mathbf{y}(0))$ and for all $t'\geq 0$ 
 $$\frac{\rmd V\big((\mathbf{p},\mathbf{q});(\mathbf{x}(t),\mathbf{y}(t)) \big)}{\rmd t} |_{t=t'}< 0.$$ 
The quantity $D_{\mathrm{KL}}(\mathbf{p} \| \mathbf{x}(t)) +  D_{\mathrm{KL}}(\mathbf{q} \| \mathbf{y}(t))$ is bounded from below by $0$ and strictly decreasing, and hence must exhibit a finite limit. Since
\begin{align*}
\frac{\rmd V \big((\mathbf{p},\mathbf{q});(\mathbf{x}(t),\mathbf{y}(t)) \big)}{\rmd t}= & -  \sum_{i} p_i \{\mathbf{Ay}(t)\}_i \\
&   -   \sum_{j} q_j \{\mathbf{Bx}(t)\}_j, 
\end{align*}
 it is immediate that  $D_{\mathrm{KL}}(\mathbf{p} \| \mathbf{x}(t)) +  D_{\mathrm{KL}}(\mathbf{q} \| \mathbf{y}(t))$ 
 has bounded second derivatives. Therefore,  Lemma \ref{eq:derivative} implies that 
 $$\lim_{t \rightarrow \infty}\frac{\rmd\big(D_{\mathrm{KL}}(\mathbf{p} \| \mathbf{x}(t)) +  D_{\mathrm{KL}}(\mathbf{q} \| \mathbf{y}(t))\big)}{\rmd t}= 0.$$
This implies by Corollary \ref{cor:KL} that  
for any $z \in \omega((\mathbf{x}(0),\mathbf{y}(0)))$ we have that the support of $z$ must be a subset of the support of $(\mathbf p, \mathbf q)$. The support of $(\mathbf p, \mathbf q)$ is not complete since $(\mathbf p, \mathbf q)$ is not a fully mixed equilibrium; hence,  $z$ must lie on the boundary as well.
Finally, since $D_{\mathrm{KL}}(\mathbf{p} \| \mathbf{x}(t)) +  D_{\mathrm{KL}}(\mathbf{q} \| \mathbf{y}(t))$ has a finite limit each $z\in\omega(\mathbf x(0), \mathbf y(0))$ must assign positive probability to all strategies in the support of $(\mathbf p, \mathbf q)$.  
\end{proof}


\section{Remarks on stochastic replicator dynamics and the generator of the Markov process}
\label{App:second_derivatives}

Due to the relation $X_1 + \dots + X_n = 1$, one might wonder if there is a dependency between the components, making $\partial_{x_j x_i} V(\mathbf{x,y})$ potentially non-zero for $j \neq i$ (and the same for $Y_1, \dots, Y_n$).
For example, when there are only two strategies $X_1$ and $X_2$ in equation~\eqref{eq:stoc2} (the same arguments also hold when considering $Y_1$, $Y_2$), one can either write the diffusion coefficients in the two components, specifically
\begin{align*}
&\sigma_1 x_1(1-x_1) \rmd W_1 + \sigma_2 x_1 x_2 \rmd W_2, \\
 &\sigma_2 x_2(1-x_2)\rmd W_2 + \sigma_1 x_1 x_2 \rmd W_1,
\end{align*}
as in equation~\eqref{eq:stoc2};
or one can reduce the problem to only one equation, say in $x_1$, with diffusion coefficient $\sqrt{\sigma_1^2 + \sigma_2^2}x_1(1-x_1)(\rmd W_1 + \rmd W_2)$ and accordingly rewritten deterministic drift, keeping track of $x_2 =1-x_1$ implicitly. Upon taking $\mathcal L (-\ln x_1 - \ln (1-x_1))$ in the latter case, one obtains exactly the same result as for the first case upon taking $\mathcal L (-\ln x_1 - \ln x_2)$, giving
\begin{align*}
&\mathcal{L} (-\ln x_1 - \ln (1-x_1)) = \mathcal L (-\ln x_1 - \ln x_2) \\
&= - [(Ay)_1 - x^\mathrm{T} Ay ] - [(Ay)_1 - x^\mathrm{T} Ay ] \\
&+ \frac{1}{2}\left( \sigma_1^2 + \sigma_2^2 \right) \left( (1-x)^2 + x_1^2 \right).
\end{align*} 
Indeed, the dependencies are already resolved by the way we write the equations;  considering the second derivatives in the other directions only matters when reducing the equations.

A more general way to see this is to transfer the approach in \cite{HofbauerImhof} to our situation, i.e.~consider 
$Z_i := X_i(Z_1 + \dots +Z_n)$ 
and $S:= Z_1 + \dots +Z_n$ in the corresponding equations and compute $\mathcal L (\ln x_i)= \mathcal L (\ln Z_i - \ln S)$, as $ Z_i $ and $S$ are not in danger to have hidden dependencies.
The corresponding equations for these variables in our case are (just compare terms with \cite{HofbauerImhof})
\begin{equation}\label{eq:Z_i}
\rmd Z_i = Z_i \left( (AY)_i + \sigma_i^2 X_i \right) \rmd t + \sigma_i Z_i \rmd W_i\,,
\end{equation}
and therefore 
\begin{equation}\label{eq:lnZ_i}
\mathcal L (-\ln Z_i) = - (AY)_i + \frac{1}{2}\sigma_i^2\left( 1 - 2 X_i \right)\,.
\end{equation}
Furthermore, we obtain
\begin{equation}\label{eq:S}
\rmd S = S \left( X^\mathrm{T}AY + \sum_{j} \sigma_j^2 X_j^2 \right) \rmd t + S \sum_j \sigma_j X_j \rmd W_j\,,
\end{equation}
and therefore
\begin{align}\label{eq:lnS}
\mathcal L (-\ln S)& = - X^\mathrm{T}AY - \sum_{j} \sigma_j^2 X_j^2 + \frac{1}{2}\sum_{j} \sigma_j^2 X_j^2 \nonumber \\
&= - X^\mathrm{T}AY - \frac{1}{2}\sum_{j} \sigma_j^2 X_j^2\,.
\end{align}
Hence, we eventually get
\begin{align}\label{eq:lnXi}
&\mathcal L (- \ln X_i ) = \mathcal{L}(-\ln Z_i) -
\mathcal L (-\ln S) \nonumber\\
&= - \left( (AY)_i - X^\mathrm{T}AY \right) + \frac{1}{2}\sigma_i^2\left( 1 - 2 X_i \right)+ \frac{1}{2}\sum_{j} \sigma_j^2 X_j^2 \nonumber \\
&= - \left( (AY)_i - X^\mathrm{T}AY \right) + \frac{1}{2}\sigma_i^2\left( 1 -  X_i \right)^2+ \frac{1}{2}\sum_{j\neq i} \sigma_j^2 X_j^2,
\end{align}
which then by summing up over all $X_i$ and also considering $Y_1, \dots, Y_n$ accordingly, leads to equation~\eqref{eq:H_bKolm}.



\section{Stochastic replicator dynamics and its connections to regret}

The (total) regret of an online learning algorithm compares its total accumulated payoff over some time horizon $[0,t]$ and compares it against the accumulated payoff of the best fixed strategy $i$ with hindsight. As long as this difference grows at most at a sublinear rate, such algorithms are known as \textit{no-regret}. It is easy to see that in the case of replicator dynamics the (total) regret remains bounded for any time interval $[0,t]$.
We can rewrite the replicator equations
\[
\dot{x}_i= x_i(u_i-\sum_j x_j u_j) 
\]
as
\[
\frac{\dot{x}_i}{x_i} = u_i-\sum_j x_j u_j,
\]
and obtain
\begin{align*}
\ln(x_i(t))-\ln(x_i(0)) 
&= \int^t_0 u_i(\tau)\rmd\tau-\int_0^t \sum_j x_j(\tau) u_j(\tau) \rmd\tau \\
&(=\text{Regret for not playing } i)
\end{align*} 
Hence, we have for all strategies $i$
\[-\ln(x_i(0)) 
\geq \int^t_0 u_i(\tau)\rmd\tau-\int_0^t \sum_j x_j(\tau) u_j(\tau) \rmd\tau,
\]
such that the regret for not having played strategy $i$ up to time $t$ is bounded by the logarithm of the initial weight chosen for that strategy.

A similar calculation can be done for the stochastic case, now using It\^{o}'s formula.
Consider now the stochastic replicator equations
\[
\rmd X_i = X_i \left(u_i - \sum_j X_j u_j\right) \rmd t + X_i (R(\mathbf{X})  \rmd W(t))_i.
\]
Then we obtain
\begin{align*}
&\ln(X_i(t))-\ln(X_i(0)) \\
&= \int^t_0 u_i(\tau)\rmd\tau-\int_0^t \sum_j X_j(\tau) u_j(\tau) \rmd\tau \\
&- \frac{1}{2} \int^t_0 \sum_j R_{i,j}^2(\mathbf{X}(\tau))\rmd\tau + \int^t_0 (R (\mathbf{X})  \rmd W(\tau))_i.
\end{align*}
Hence, taking expectations, we obtain
\begin{align*}
&-\ln(X_i(0)) + \frac{1}{2} \sum_j \int^t_0 \mathbb{E} \left[R_{i,j}^2(\mathbf{X}(\tau)) \right] \rmd \tau \\
&\geq \int^t_0 \mathbb{E}\left[u_i(\tau)\right]\rmd\tau-\int_0^t  \mathbb{E}\left[ \sum_j X_j(\tau) u_j(\tau) \right] \rmd\tau.
\end{align*}
For model~\eqref{eq:stoc2}, we find that the additional term in the estimate of expected regret is given by
\begin{align*}
&\frac{1}{2} \sum_j \int^t_0 \mathbb{E} \left[R_{i,j}^2(\mathbf{X}(\tau)) \right] \rmd \tau  \\
&= \frac{1}{2} \int^t_0 \sum_{j \neq i} \sigma_j^2 \mathbb{E} \left[X_j(\tau)^2 \right] + \sigma_i^2 \mathbb{E}\left[(1-X_i(\tau))^2\right] \rmd \tau \\
&\leq \frac{1}{2} \int^t_0 \sum_{j \neq i} \sigma_j^2 \mathbb{E} \left[X_j(\tau)^2 \right] \\
&+ \sigma_i^2 \left(\sum_{j \neq i} \mathbb{E}\left[X_j^2(\tau)\right] + \sum_{j,k \neq i, j\neq k} \mathbb{E}\left[X_j(\tau) X_k(\tau)\right] \right)  \rmd \tau\\
& \leq C \max_{k} \sigma_k^2 \, t,
\end{align*} 
where $2 > C > 0$ is a constant.

Given the particularly strong no-regret properties of replicator dynamics as well as other deterministic continuous-time Follow-the-Regularized Leader (FTRL) dynamics~\cite{sorin2009exponential,GeorgiosSODA18}, the above characterization is not particularly surprising. Nevertheless, using the standard approximate regret minimization techniques, it suffices to argue that the time-average of stochastic replicator dynamics converge to an approximate Nash equilibrium~\cite{Cesa06}. Naturally, this characterization constrains our invariant measures even further, however, exploring the full implications of this connection is beyond the scope of the current work.



\end{document}